\documentclass{article}
\usepackage{amssymb}

\usepackage{amsfonts}
\usepackage{amsmath}

\newtheorem{theorem}{Theorem}

\newtheorem{corollary}[theorem]{Corollary}

\newtheorem{definition}[theorem]{Definition}

\newtheorem{lemma}[theorem]{Lemma}

\newtheorem{remark}[theorem]{Remark}

\newenvironment{proof}[1][Proof]{\noindent\textbf{#1.} }{\ \rule{0.5em}{0.5em}}
\input{tcilatex}

\begin{document}

\title{Model problem for integro-differential Zakai equation with
discontinuous observation processes in H\"{o}lder spaces}
\author{R.Mikulevicius and H.Pragarauskas \\
University of Southern California, Los Angeles\\
Institute of Mathematics and Infrmatics, Vilnius}
\maketitle

\begin{abstract}
The existence and uniqueness of solutions of the Cauchy problem to a a
stochastic parabolic integro-differential equation is investigated. The
equattion considered arises in nonlinear filtering problem with a jump
signal process and jump observation.
\end{abstract}

\section{Introduction}

In a complete probability space $(\Omega ,\mathcal{F},\mathbf{P})$ with a
filtration of $\sigma $-algebras $\mathbb{F}=(\mathcal{F}_{t})$ satisfying
the usual conditions, we consider the linear stochastic integro-differential
parabolic equation 
\begin{equation}
\left\{ 
\begin{array}{ll}
du(t,x)=\bigl(A^{(\alpha )}u(t,x)+f(t,x)\bigr)dt+\int_{U}g(t,x,v)q(dt,dv) & 
\text{in }H, \\ 
u(0,x)=0 & \text{in }\mathbf{R}^{d}%
\end{array}%
\right.  \label{intr1}
\end{equation}%
of the order $\alpha \in (0,2]$, where $H=\left[ 0,T\right] \times \mathbf{R}%
^{d},\,q(dt,d\upsilon )$ is a martingale measure on a measurable space $%
([0,\infty )\times U,\mathcal{B}([0,\infty ))\otimes \mathcal{U}),\,g$ is an 
$\mathbb{F}$-adapted measurable real-valued function on $H\times U,f$ is an $%
\mathbb{F}$-adapted measurable real-valued function on $H$, 
\begin{eqnarray}
A^{(\alpha )}u(t,x) &=&\int_{\mathbf{R}_{0}^{d}}[u(t,x+y)-u(t,x)-(\nabla
u(t,x),y)\chi ^{(\alpha )}{(y)}]m^{(\alpha )}(t,y)\frac{dy}{|y|^{d+\alpha }}
\notag \\
&&+\bigl(b(t),\nabla u(t,x)\bigr)1_{\alpha =1}+\frac{1}{2}%
\sum_{i,j=1}^{d}B^{ij}(t)\partial _{ij}^{2}u(t,x)1_{\alpha =2},  \label{two}
\\
\chi ^{(\alpha )}{(y)} &=&1_{\alpha >1}+1_{|y|\leqslant 1}1_{\alpha =1}, 
\notag
\end{eqnarray}%
$m^{(\alpha )}(t,y)$ is a deterministic measurable real-valued function
homogeneous in $y$ of order zero, $m^{(2)}=0,\mathbf{R}_{0}^{d}=\mathbf{R}%
^{d}\backslash \{0\}$ and $b(t)=(b^{1}(t),\ldots ,b^{d}(t)),B(t)=(B^{ij}(t))$
are deterministic bounded measurable function$.$ It is the model problem for
the Zakai equation (see \cite{za}) arising in the nonlinear filtering
problem. Assume that the signal process $X_{t}$ is defined by 
\begin{eqnarray*}
X_{t} &=&X_{0}+\int_{0}^{t}b(s)1_{\alpha =1}ds+\int_{0}^{t}\sqrt{B(s)}%
dW_{s}1_{\alpha =2} \\
&&+\int_{0}^{t}\int \chi ^{(\alpha )}{(y)}y\widetilde{q}(ds,dy)+\int_{0}^{t}%
\int (1-\chi ^{(\alpha )}{(y)})y\widetilde{p}(ds,dy),
\end{eqnarray*}%
where $\widetilde{p}(ds,dy)$ is a point measure on $[0,\infty )\times 
\mathbf{R}_{0}^{d}$ with a compensator $m^{(\alpha )}(s,y)\frac{dyds}{%
|y|^{d+\alpha }},$%
\begin{equation*}
\widetilde{q}(ds,dy)=\widetilde{p}(ds,dy)-m^{(\alpha )}(s,y)\frac{dyds}{%
|y|^{d+\alpha }} 
\end{equation*}%
and $W_{t}$ is a standard Wiener process in $\mathbf{R}^{d}$.

Assume $X_{0}$ has a probability density function $u_{0}\left( x\right) $
and the observation $Y_{t}$ is discontinuous, with jump intensity depending
on the signal, such that%
\begin{equation*}
Y_{t}=\int_{0}^{t}\int_{|y|>1}yp(ds\,dy)+\int_{0}^{t}\int_{|y|\leqslant 1}y%
\hat{q}(ds,dy), 
\end{equation*}%
where $p(ds,dy)$ is a point measure on $[0,\infty )\times \mathbf{R}_{0}^{d}$
not having common jumps with $\widetilde{p}(ds,dy)$ with a compensator $\rho
(X_{t},y)\pi (dy)$ and $\hat{q}(dt,dy)=p(dt,dy)-\pi (dy)dt$. Assume $%
C_{1}\geqslant \rho (x,y)\geqslant c_{1}>0,\pi (dy)$ is a measure on $%
\mathbf{R}_{0}^{d}$ such that%
\begin{equation*}
\int |y|^{2}\wedge 1\pi (dy)<\infty , 
\end{equation*}%
and $\int [\rho (x,y)-1]^{2}\pi (dy)$ is bounded. Then for every function $%
\psi $ such that $\mathbf{E}[\psi (X_{t})^{2}]<\infty ,$ the optimal mean
square estimate for $\psi \left( X_{t}\right) ,\,t\in \left[ 0,T\right] $,
given the past of \ the observations $\mathcal{F}_{t}^{Y}=\sigma
(Y_{s},s\leqslant t),$ is of the form 
\begin{equation*}
\hat{\psi}_{t}=\mathbf{E}\bigl[\psi (X_{t})|\mathcal{F}_{t}^{Y}\bigr]=\frac{%
\mathbf{\widetilde{E}}\big[\psi (X_{t})\zeta _{t}|\mathcal{F}_{t}^{Y}\big]}{%
\mathbf{\widetilde{E}}\big[\zeta _{t}|\mathcal{F}_{t}^{Y}\big ]}, 
\end{equation*}%
where $\zeta _{t}$ is the solution of the linear equation%
\begin{equation*}
d\zeta _{t}=\zeta _{t-}\int [\rho (X_{t-},y)-1]\hat{q}(dt,dy) 
\end{equation*}%
and $d\mathbf{\widetilde{P}}=\zeta \left( T\right) ^{-1}d\mathbf{P}.$ Under
some assumptions, one can easily show that if $v(t,x)$ is an $\mathbb{F=(}%
\mathcal{F}_{t+}^{Y})$-adapted unnormalized filtering density function 
\begin{equation}
\mathbf{\widetilde{E}}\left[ \psi \left( X_{t}\right) \zeta _{t}|\mathcal{F}%
_{t}^{Y}\right] =\int v\left( t,x\right) \psi \left( x\right) \,dx,
\label{prf0}
\end{equation}%
then it is a solution of the Zakai equation 
\begin{eqnarray}
dv(t,x) &=&v(t,x)\int [\rho (x,y)-1]\hat{q}(dt,dy)+\bigg\{\bigl(b(t),\nabla
v(t,x)\bigr)1_{\alpha =1}  \notag \\
&&+\frac{1}{2}\sum_{i,j=1}^{d}B^{ij}(t)\partial _{ij}^{2}v(t,x)1_{\alpha =2}
\notag \\
&&+\int_{\mathbf{R}_{0}^{d}}[v(t,x+y)-v(t,x)-(\nabla v(t,x),y)\chi ^{(\alpha
)}{(y)}]m^{(\alpha )}(t,-y)\frac{dy}{|y|^{d+\alpha }}\bigg\},  \notag \\
v(0,x) &=&u_{0}(x).  \label{prf1}
\end{eqnarray}%
Since $Y_{t},t\geqslant 0,$ and $X_{t},t\geqslant 0,$ are independent with
respect to $\mathbf{\widetilde{P}},$ for $u\left( t,x\right) =v\left(
t,x\right) -u_{0}\left( x\right) $ we have an equation whose model problem
is of the type given by (\ref{intr1}). Indeed, according to \cite{grig1},
for any infinitely differentiable function $\psi $ on $\mathbf{R}^{d}$ with
compact support, the conditional expectation $\pi _{t}(\psi )=\mathbf{%
\widetilde{E}}\left[ \psi \left( X_{t}\right) \zeta _{t}|\mathcal{F}_{t}^{Y}%
\right] $ satisfies the equation%
\begin{eqnarray*}
d\pi _{t}(\psi ) &=&\int \pi _{t}\big(\psi \lbrack \rho (\cdot ,y)-1]\big)%
\hat{q}(dt,dy)+\pi _{t}\bigg\{(b(t),\nabla \psi )1_{\alpha =1} \\
&&+\frac{1}{2}\sum_{i,j=1}^{d}B^{ij}(t)\partial _{ij}^{2}\psi (t,x)1_{\alpha
=2} \\
&&+\int_{\mathbf{R}_{0}^{d}}\big[\psi (\cdot +y)-\psi -(\nabla \psi ,y)\chi
^{(\alpha )}(y)\big]m^{(\alpha )}(t,y)\frac{dy}{|y|^{d+\alpha }}\bigg\}dt.
\end{eqnarray*}%
Assuming (\ref{prf0}) and integrating by parts, we obtain (\ref{prf1}).

The general Cauchy problem for a linear parabolic SPDE of the second order 
\begin{equation}
\left\{ 
\begin{array}{ll}
du=(a^{ij}\,\partial _{ij}u+b^{i}\partial _{i}u+c\,u+f)dt+(\sigma
^{i}\partial _{i}u+h\,u+g)d{W}_{t} & \text{in }H, \\ 
u(0,x)=0 & \text{in }\mathbf{R}^{d}%
\end{array}%
\right.  \label{intr3}
\end{equation}%
driven by a Wiener process $W_{t}$ has been studied by many authors. When
the matrix $(2a^{ij}-\sigma ^{i}\cdot \sigma ^{j})$ is uniformly
non-degenerate there exists a complete theory in Sobolev spaces $%
W^{n,2}\left( \mathbf{R}^{d}\right) $ (see \cite{par}, \cite{kryro1}, \cite%
{ro2} and references therein) and in the spaces of Bessel potentials $%
H_{s}^{p}\left( \mathbf{R}^{d}\right) $ (see \cite{kry1}). The equation (\ref%
{intr3}) in H\"{o}lder classes with $\sigma =0$ was considered in \cite{ro}, 
\cite{m}. The equation (\ref{intr3}) in H\"{o}lder classes with $\sigma =0$
and integro-differential operator of the order $\alpha $ in the drift part
was studied in \cite{mip2}.

In this paper, following the main steps in \cite{mip}, \cite{mip2}, we prove
the solvability of the Cauchy problem (\ref{intr1}). In Section 2, we
introduce the notation and state our main result. In Section 3, we prove
some auxiliary results concerning moment estimates of discontinuous
martingales and the solvability of (\ref{intr1}) for smooth input functions.
The proof of the main theorem is given in Section 4.

\section{Notation and main result}

\subsection{Notation}

The following notation will be used in the paper.

Let $(\Omega ,\mathcal{F},\mathbf{P})$ be a complete probability space with
a right-continuous filtration of $\sigma $-algebras $\mathbb{F}=(\mathcal{F}%
_{t})_{t\in \lbrack 0,T]},$ $\mathcal{F}_{0}$ containing all $\mathbf{P}$%
-null sets of $\mathcal{F}$. Let $(U,\mathcal{U})$ be a measurable space
with a $\sigma $-finite non-negative measure $\Pi (d\upsilon )$. Throughout
the paper we assume that there is an increasing sequence $U_n\subset U$ such
that 
\begin{equation*}
U=\cup_n U_n;\quad \Pi(U_n)<\infty,\qquad n=1,2,\ldots . 
\end{equation*}

We say that a measurable function $f:[0,T]\times \Omega \times \mathbf{R}%
^{d}\rightarrow \mathbf{R}$ is $\mathbb{F}$-progressively measurable if for
each $t\in \lbrack 0,T]$ the mapping $(s,\omega ,x)\rightarrow f(s,\omega
,x) $ on $[0,t]\times \Omega \times \mathbf{R}^{d}$ is $\mathcal{B}%
([0,t])\otimes \mathcal{F}_{t}\otimes \mathcal{B}(\mathbf{R}^{d})$%
-measurable.

We denote $\mathcal{P}(\mathbb{F})$ the $\mathbb{F}$-predictable $\sigma $%
-algebra on $[0,T]\times \Omega $ and $\mathcal{O}(\mathbb{F})$ the $\sigma $%
-algebra of $\mathbb{F}$-well measurable sets on $[0,T]$. We denote $%
\mathcal{R}(\mathbb{F})$ the $\mathbb{F}$-progressive $\sigma $-algebra on $%
[0,T]\times \Omega $.

We say that a stochastic process $X_{t},0\leqslant t\leqslant T,$ is cadlag
if $\mathbf{P}$-a.s. it is right-continuous ($X_{t+}=X_{t})$, the left-hand
limits $X_{t-}$ exist for all $t\in \lbrack 0,T]$ and $X_{T-}=X_{T}.$

We denote $H=[0,T]\times \mathbf{R}^{d},\mathbf{N}_{0}=\{0,1,2,\ldots \},%
\mathbf{R}_{0}^{d}=\mathbf{R}^{d}\backslash \{0\}.$ If $x,y\in \mathbf{R}%
^{d},$\ we write 
\begin{equation*}
(x,y)=\sum_{i=1}^{d}x_{i}y_{i},\,|x|=\sqrt{(x,x)}. 
\end{equation*}

Let $L_{p}(\Omega ,\mathcal{F},\mathbf{P})$, $p\geqslant 1$, be the space of
random variables $X$ with finite norm 
\begin{equation*}
|X|_{p}=\bigl(\mathbf{E}|X|^{p}\bigr)^{\frac{1}{p}}. 
\end{equation*}

Let $B_{p}(H)$, $p\geqslant 1$, be the space of all $\mathcal{R}(\mathbb{F}%
)\otimes\mathcal{B}(\mathbf{R}^d)$-measurable functions $u\colon\ \Omega
\times H\to \mathbf{R}$ such that 
\begin{equation*}
||u||_{p}=\sup_{(t,x)\in H} \vert u(t,x) \vert _p <\infty . 
\end{equation*}
Similarly, $B_{r,p}(H\times U),\ r\geqslant 1,\ p\geqslant 1$, is the space
of all $\mathcal{R}(\mathbb{F})\otimes\mathcal{B}(\mathbf{R}^d)\otimes 
\mathcal{U}$-measurable functions $g\colon\ \Omega \times H\times U\to 
\mathbf{R}$ such that 
\begin{equation*}
\Vert g\Vert_{r,p} = \sup_{(t,x)\in H} \big\vert g(t,x,\cdot)\big\vert %
_{r,p} < \infty , 
\end{equation*}
where 
\begin{equation*}
\vert g(t,x,\cdot)\vert _{r,p} = \bigg\vert \biggl( \int_U \vert
g(t,x,v)\vert ^r \Pi(dv)\biggr) ^{\frac{1}{r}} \bigg\vert _p . 
\end{equation*}

For $L_{p}(\Omega ,\mathcal{F},\mathbf{P})$-valued function $u$ on $H$ or $%
\mathbf{R}^{d},$ we denote its partial derivatives in $x$ in $L_{p}(\Omega ,%
\mathcal{F},\mathbf{P})$-sense by $\partial _{i}u=\partial u/\partial x_{i}$%
, $\partial _{ij}^{2}u=\partial ^{2}u/\partial x_{i}\partial x_{j}$, etc.;$%
\,\partial u=\nabla u=(\partial _{1}u,\ldots ,\partial _{d}u)$ denotes the
gradient of $u$ with respect to $x$; for a multiindex $\gamma \in \mathbf{N}%
_{0}^{d}$ we denote%
\begin{equation*}
\partial _{x}^{\gamma }u(t,x)=\frac{\partial ^{|\gamma |}u(t,x)}{\partial
x_{1}^{\gamma _{1}}\ldots \partial x_{d}^{\gamma _{d}}}. 
\end{equation*}%
For $\alpha \in (0,2]$, we write 
\begin{equation*}
\partial ^{\alpha }u(t,x)=\mathcal{F}^{-1}[|\xi |^{\alpha }\mathcal{F}%
u(t,\xi )](x), 
\end{equation*}%
where $\mathcal{F}$ denotes the Fourier transform with respect to $x\in 
\mathbf{R}^{d}$ and $\mathcal{F}^{-1}$ is the inverse Fourier transform,
i.e. 
\begin{equation*}
\mathcal{F}u(t,\xi )=\int_{\mathbf{R}^{d}}\,\mathrm{e}^{-i(\xi
,x)}u(t,x)dx,\quad \mathcal{F}^{-1}u(t,\xi )=\frac{1}{(2\pi )^{d}}\int_{%
\mathbf{R}^{d}}\,\mathrm{e}^{i(\xi ,x)}u(t,\xi )d\xi . 
\end{equation*}

For $\alpha \in \lbrack 0,2],\ \beta \in (0,1),\ C_{p}^{\alpha ,\beta }(H)$
is the set of all $u\in B_{p}(H)$ with finite norm 
\begin{equation*}
\Vert u\Vert _{\alpha ,\beta ;p}=\Vert u\Vert _{p}+\Vert \partial ^{\alpha
}u\Vert _{p}+[\partial ^{\alpha }u]_{\beta ;p}, 
\end{equation*}%
where 
\begin{equation*}
\lbrack u]_{\beta ;p}=\sup_{t,x\neq y}\frac{|u(t,x)-u(t,y)|_{p}}{%
|x-y|^{\beta }}. 
\end{equation*}

Similarly, $C^{\alpha ,\beta }_{r,p}(H\times U)$ is the set of all $g\in
B_{r,p}(H\times U)$ with finite norm 
\begin{equation*}
\Vert g\Vert _{\alpha ,\beta ;r,p} =\Vert g\Vert_{r,p} +\Vert
\partial^{\alpha }g\Vert_{r,p} +[\partial^{\alpha }u] _{\beta ;r,p} , 
\end{equation*}
where 
\begin{equation*}
[g]_{\beta ;r,p} =\sup_{t,x\neq y} \frac{\vert
g(t,x,\cdot)-g(t,y,\cdot)\vert_{r,p} }{\vert x-y\vert^{\beta } } . 
\end{equation*}

For $\alpha =0$, we write 
\begin{equation*}
\Vert u\Vert_{0,\beta ;p}= \Vert u\Vert_p + [u]_{\beta ;p} 
\end{equation*}
and 
\begin{equation*}
\Vert g\Vert_{0,\beta ;r,p}= \Vert g\Vert_{r,p} + [g]_{\beta ;r,p} . 
\end{equation*}

$C^{\infty}_p(H)$ is the set of all $u\in B_p(H)$ such that $\mathbf{P}$%
-a.s. for all $t\in[0,T]$ the function $u(t,x)$ is infinitely differentiable
in $x$ and for every multiindex $\gamma \in \mathbf{N}^d_0$ 
\begin{equation*}
\sup_{(t,x)\in H} \big\vert \partial^{\gamma }_x u(t,x)\big\vert _p < \infty
. 
\end{equation*}
Similarly, $C^{\infty}_{r,p}(H\times U)$ is the set of all $g\in
B_{r,p}(H\times U)$ such that $\mathbf{P}$-a.s. for all $t\in[0,T]$, $v\in U$
the function $g(t,x,v)$ is infinitely differentiable in $x$ and for every
multiindex $\gamma \in \mathbf{N}^d_0$ 
\begin{equation*}
\sup_{(t,x)\in H} \big\vert \partial^{\gamma }_x g(t,x,\cdot)\big\vert %
_{r,p} < \infty . 
\end{equation*}

The counterparts of spaces $C^{\alpha ,\beta }_p(H)$ and $C^{\infty}_p(H)$
for nonrandom functions are denoted simply by $C^{\alpha ,\beta }(H)$ and $%
C^{\infty}(H)$. We denote $C^{\infty}_0(\mathbf{R}^d)$ the set of all
infinitely differentiable functions on $\mathbf{R}^d$ with compact support.

$C=C(\cdot ,\ldots ,\cdot ),c=c(\cdot ,\ldots ,\cdot )$ denote constants
depending only on quantities appearing in parentheses. In a given context
the same letter will (generally) be used to denote different constants
depending on the same set of arguments.

\subsection{\noindent Main result}

Let $\alpha \in (0,2]$ and $b\colon \ [0,T]\rightarrow \mathbf{R}^{d}$, $%
m^{(\alpha )}\colon \ [0,T]\times \mathbb{R}_{0}^{d}\rightarrow \mathbf{R}$
be measurable functions. Throughout the paper we assume that the function $%
m^{(\alpha )}(t,y)$ is homogeneous in $y$ of order zero, i.e. for all $t\in
\lbrack 0,T],\ r>0,\ y\in \mathbf{R}_{0}^{d}$ 
\begin{equation*}
m^{(\alpha )}(t,ry)=m^{(\alpha )}(t,y), 
\end{equation*}%
and 
\begin{equation*}
\int_{S^{d-1}}wm^{(1)}(t,w)\mu _{d-1}(dw)=0, 
\end{equation*}%
where $S^{d-1}$ is the unit sphere in $\mathbf{R}^{d}$ and $\mu _{d-1}$ is
the Lebesgue measure on it. Also, $m^{(2)}=0.$

Let $A^{(\alpha )},\ \alpha \in (0,2)$, be the operators defined by (\ref%
{two}). In terms of Fourier transform%
\begin{equation*}
A^{(\alpha )}u(x)=\mathcal{F}^{-1}\left[ \psi ^{(\alpha )}(t,\xi )\mathcal{F}%
u(\xi )\right] (x), 
\end{equation*}%
where%
\begin{eqnarray*}
{}\psi ^{(\alpha )}(t,\xi ) &=&i(b(t),\xi )1_{\alpha =1}+\frac{1}{2}%
\sum_{i,j=1}^{d}B^{ij}(t)\xi _{i}\xi _{j}1_{\alpha =2} \\
&&-C\int_{S^{d-1}}|(w,\xi )|^{\alpha }\Big[1-i\Big(\tan \frac{\alpha \pi }{2}%
\mbox{sgn}(w,\xi )1_{\alpha \neq 1} \\
&&-\frac{2}{\pi }\mbox{sgn}(w,\xi )\ln |(w,\xi )|1_{\alpha =1}\Big)\Big]%
m^{(\alpha )}(t,w)\mu _{d-1}(dw) \\
&=&-|\xi |^{\alpha }M^{(\alpha )}(t,\xi ),
\end{eqnarray*}%
$C=C(\alpha ,d)$ is a positive constant and 
\begin{eqnarray}
M^{(\alpha )}(t,\xi ) &=&-i(b(t),\frac{\xi }{|\xi |})1_{\alpha =1}+\frac{1}{2%
}\sum_{i,j=1}^{d}B^{ij}(t)\xi _{i}\xi _{j}|\xi |^{-2}1_{\alpha =2}  \notag \\
&&+C\int_{S^{d-1}}\Big|\Big(w,\frac{\xi }{|\xi |}\Big)\Big|^{\alpha }\bigg[%
1-i\bigg(\tan \frac{\alpha \pi }{2}\mathop{{\mathrm{sgn}}}\Big(w,\frac{\xi }{%
|\xi |}\Big)1_{\alpha \neq 1}  \label{M} \\
&&-\frac{2}{\pi }\mathop{{\mathrm{sgn}}}\Big(w,\frac{\xi }{|\xi |}\Big)\ln %
\Big|\Big(w,\frac{\xi }{|\xi |}\Big)\Big|1_{\alpha =1}\bigg)\bigg]m^{(\alpha
)}(t,w)\mu _{d-1}(dw)  \notag
\end{eqnarray}%
is $0$-homogeneous function with respect to $\xi $.

Let $p(dt,d\upsilon )$ be an $\mathbb{F}$-adapted Poisson point measure on $%
[0,\infty )\times U$ with a compensator $\Pi (d\upsilon )dt$ and 
\begin{equation*}
q(dt,dv) = p(dt,dv) - \Pi(dv)dt 
\end{equation*}
be a martingale measure.

In stochastic H\"{o}lder spaces ${C}_{p}^{\alpha ,\beta }(H),$ we consider
the Cauchy problem%
\begin{equation}
\left\{ 
\begin{array}{ll}
du(t,x)=\big[(A^{(\alpha )}-\lambda )u(t,x)+f(t,x)\big]dt+\int_{U}g(t,x,%
\upsilon )q(dt,d\upsilon ) & \text{in }H, \\ 
u(0,x)=0 & \text{in }\mathbf{R}^{d},%
\end{array}%
\right.  \label{eq1}
\end{equation}%
where $\lambda \geqslant 0$.

We will need the following assumptions.

\textbf{A1}. There is a constant $\mu >0$ such that for all $t\in \lbrack
0,T]$%
\begin{equation*}
\inf_{|\xi |=1}\,\mathrm{Re}\,M^{(\alpha )}(t,\xi )\geq \mu ; 
\end{equation*}

\textbf{A2.} The function $m^{(\alpha )}(t,y)$ is differentiable in $y$ up
to the order $d_{0}$ and 
\begin{equation*}
C^{(\alpha )}=\sup_{0\leq t\leq T}\big[\sup_{|\gamma |\leq
d_{0},|y|=1}|\partial _{y}^{\gamma }m^{(\alpha )}(t,y)|+|b(t)|1_{\alpha
=1}+|B(t)|1_{\alpha =2}\big]<\infty , 
\end{equation*}%
where $d_{0}=[\frac{d}{2}]+1$ and $[\frac{d}{2}]$ is the integer part of $%
\frac{d}{2}$.

\begin{remark}
{Assumption \textbf{A1} is satisfied if and only if for all $t\in \lbrack
0,T],\xi \in \mathbf{R}^{d},|\xi |=1,$ 
\begin{eqnarray*}
(B(t)\xi ,\xi ) &\geq &\mu ,~\alpha =2, \\
\int_{S^{d-1}}|(w,\xi )|^{\alpha }m^{(\alpha )}(t,w)\mu _{d-1}(dw) &\geq
&\mu ,\quad \alpha \in (0,2).
\end{eqnarray*}%
The last condition holds with some constant $\mu >0$ if, for example, there
is a Borel set $\Gamma \subseteq S^{d-1}$ such that $\mu _{d-1}(\Gamma )>0$
and} 
\begin{equation*}
\inf_{t\in \lbrack 0,T],w\in \Gamma }m^{(\alpha )}(t,w)>0. 
\end{equation*}
\end{remark}

\begin{definition}
Let $\alpha \in (0,2],\ \beta \in (0,1),\ p\geq 2,$ $f\in B_{p}(H),\ g\in
B_{l,p}(H\times U),l=2,p.$ We say that $u\in {C}_{p}^{\alpha ,\beta }(H)$ is
a solution of (\ref{eq1}) if for each $(t,x)\in H$ $\mathbf{P}$-a.s.%
\begin{equation*}
u(t,x)=\int_{0}^{t}\big[A^{(\alpha )}u(s,x)-\lambda u(s,x)+f(s,x)\big]%
ds+\int_{0}^{t}\int_{U}g(s,x,\upsilon )q(ds,d\upsilon ).\label{defs} 
\end{equation*}
\end{definition}

Now we state the main result of this paper.

\begin{theorem}
\label{main}Let $\alpha ,\alpha ^{\prime }\in (0,2],\ \beta ,\beta ^{\prime
}\in (0,1)$, $p\geqslant 2$ and assumptions \textbf{A1, A2} be satisfied.
Assume that $\alpha (1-1/p)+\beta =\alpha ^{\prime }+\beta ^{\prime }$, $%
f\in C_{p}^{0,\beta }(H)$ and $g\in C_{r,p}^{\alpha ^{\prime },\beta
^{\prime }}(H\times U)$, $r=2,p$.

Then there is a unique solution $u\in {C}_{p}^{\alpha ,\beta }(H)$ to (\ref%
{eq1}). Moreover, there is a constant $C$ depending only on $\alpha ,\beta
,p,d,\mu ,C^{(\alpha )},T$ such that the following estimates hold:

(i) 
\begin{equation*}
\Vert u\Vert _{\alpha ,\beta ;p}\leqslant C\bigg(\Vert f\Vert _{0,\beta
;p}+\sum_{r=2,p}\Vert g\Vert _{\alpha ^{\prime },\beta ^{\prime };r,p}\bigg)%
, 
\end{equation*}

(ii) 
\begin{equation*}
\Vert u\Vert _{0,\beta ;p}\leq C\Bigl(T\wedge \frac{1}{\lambda }\Bigr)^{%
\frac{1}{p}}\biggl(\Vert f\Vert _{0,\beta ;p}+\sum_{r=2,p}\Vert g\Vert
_{0,\beta ;r,p}\biggr); 
\end{equation*}

(iii) for $0\leq t\leq t^{\prime }\leq T$ 
\begin{equation*}
\Vert u(t^{\prime },\cdot )-u(t,\cdot )\Vert _{\alpha ^{\prime },\beta
^{\prime };p}\leq C(t^{\prime }-t)^{\frac{1}{p}}\bigg(\Vert f\Vert _{0,\beta
;p}+\sum_{r=2,p}\Vert g\Vert _{\alpha ^{\prime }\beta ^{\prime };r,p}\bigg). 
\end{equation*}
\end{theorem}

\begin{remark}
{If $\alpha ^{\prime }=\alpha (1-1/p)$, then $\beta ^{\prime }=\beta .$
Lemma \ref{lem4} below indicates that we could assume this without any loss
of generality.}

{If $p=2,$ one can take $\alpha ^{\prime }=\alpha /2$ and $\beta ^{\prime
}=\beta $. In this case the estimates of Theorem \ref{main} are similar to
the corresponding estimates of Lemma 17 \cite{mip2} with $p=2$ for the Zakai
equation driven by a Wiener process.}
\end{remark}

\section{Auxiliary results}

\subsection{Moment estimates}

First we prove some discontinuous martingale moment estimates. Denote $%
L_{loc}^{2,p}$ the space of all $\mathcal{R}(\mathbb{F})\otimes \mathcal{U}$%
-measurable functions $g(t,\upsilon )=g(\omega ,t,\upsilon )$ such that $%
\mathbf{P}$-a.s.%
\begin{equation*}
\int_{0}^{T}\int_{U}|g(t,\upsilon )|^{p}\Pi (d\upsilon
)dt+\int_{0}^{T}\int_{U}g(t,\upsilon )^{2}\Pi (d\upsilon )dt<\infty . 
\end{equation*}

\begin{lemma}
\label{hl0}Let $p\geq 2,g\in L_{loc}^{2,p}$ and%
\begin{equation*}
Q_{t}=\int_{0}^{t}\int_{U}g(s,\upsilon )q(ds,d\upsilon ),0\leq t\leq T. 
\end{equation*}%
Then there is a constant $C=C(p)$ such that for any $\mathbb{F}$-stopping
time $\tau \leq T,$ 
\begin{eqnarray}
\mathbf{E}\big[\sup_{t\leq \tau }|Q_{t}|^{p}\big] &\leq &C\mathbf{E}\bigg[%
\int_{0}^{\tau }\int_{U}|g(s,\upsilon )|^{p}\Pi (d\upsilon )ds+  \notag \\
&&+\bigg(\int_{0}^{\tau }\int_{U}g(s,\upsilon )^{2}\Pi (d\upsilon )ds\bigg)%
^{p/2}\bigg]  \label{eight}
\end{eqnarray}%
Moreover, if $\sup_{s,\upsilon }|g(s,\upsilon )|<\infty $ $\mathbf{P}$-a.s.,
then for each $\varepsilon >0$ there is a constant $C(\varepsilon ,p)$ such
that%
\begin{eqnarray}
\mathbf{E}\big[\sup_{t\leq \tau }|Q_{t}|^{p}\big] &\leq &\varepsilon \mathbf{%
E}\Big[\sup_{0\leq s\leq \tau ,\upsilon }|g(s,\upsilon )|^{p}\Big]+  \notag
\\
&&+C({\varepsilon },p)\mathbf{E}\bigg[\bigg(\int_{0}^{\tau
}\int_{U}g(s,\upsilon )^{2}\Pi (d\upsilon )ds\bigg)^{p/2}\bigg].
\label{nine}
\end{eqnarray}
\end{lemma}

\begin{proof}
Let%
\begin{equation*}
A_{t}=\int_{0}^{t}\int_{U}g(s,\upsilon )^{2}p(ds,d\upsilon ),\quad
L_{t}=\int_{0}^{t}\int_{U}g(s,\upsilon )^{2}\Pi (d\upsilon )ds,\quad t\geq
0. 
\end{equation*}%
By the Burkholder--Davis--Gundy inequality, for each $\mathbb{F}$-stopping
time $\tau $%
\begin{equation*}
\mathbf{E}\big[\sup_{t\leq \tau }|Q_{t}|^{p}\big]\leq C_{p}\mathbf{E}%
[A_{\tau }^{p/2}]. 
\end{equation*}%
Denoting $q=p/2\geq 1$, we have%
\begin{equation*}
A_{\tau }^{q}=\sum_{s\leq \tau }\big[(A_{s-}+\Delta A_{s})^{q}-A_{s-}^{q}%
\big]=\int_{0}^{\tau }\int_{U}\big[(A_{s-}+g(s,\upsilon )^{2})^{q}-A_{s-}^{q}%
\big]p(ds,d\upsilon ), 
\end{equation*}%
and%
\begin{equation*}
\mathbf{E}[A_{\tau }^{q}]=\mathbf{E}\int_{0}^{\tau }\int_{U}\big[%
(A_{s-}+g(s,\upsilon )^{2})^{q}-A_{s-}^{q}\big]\Pi (d\upsilon )ds. 
\end{equation*}%
Since there are two positive constants $c_{q},C_{q}$ such that for all
non-negative numbers $a,b$%
\begin{equation*}
C_{q}\bigl(b^{q}+a^{q-1}b\bigr)\geq (a+b)^{q}-a^{q}\geq c_{q}\bigl(%
b^{q}+a^{q-1}b\bigr), 
\end{equation*}%
we have%
\begin{eqnarray}
&&C_{q}\mathbf{E}\int_{0}^{\tau }\int_{U}\big[|g(s,\upsilon
)|^{p}+A_{s-}^{q-1}g(s,\upsilon )^{2}\big]\Pi (d\upsilon )ds\geq \mathbf{E}%
[A_{\tau }^{q}]  \label{hf0} \\
&&\qquad \geq c_{q}\mathbf{E}\int_{0}^{\tau }\int_{U}\big[|g(s,\upsilon
)|^{p}+A_{s-}^{q-1}g(s,\upsilon )^{2}\big]\Pi (d\upsilon )ds,  \notag
\end{eqnarray}%
Hence, 
\begin{equation*}
\mathbf{E}[A_{\tau }^{q}]\leq C_{q}\bigg\{\mathbf{E}\int_{0}^{\tau
}\int_{U}|g(s,\upsilon )|^{p}\Pi (d\upsilon )ds+\mathbf{E[}A_{\tau
}^{q-1}L_{\tau }]\bigg\}. 
\end{equation*}%
According to Young's inequality, for each $\varepsilon >0$ there is a
constant $C_{\varepsilon }$ such that 
\begin{equation*}
A_{\tau }^{q-1}L_{\tau }\leq \varepsilon A_{\tau }^{q}+C_{\varepsilon
}L_{\tau }^{q}. 
\end{equation*}%
Therefore, there is a constant $C$ such that 
\begin{equation*}
\mathbf{E}[A_{\tau }^{q}]\leq \frac{1}{2}\mathbf{E}[A_{\tau }^{q}]+C\mathbf{E%
}\bigg(\int_{0}^{\tau }\int_{U}|g(s,\upsilon )|^{p}\Pi (d\upsilon
)ds+L_{\tau }^{q}\bigg)
\end{equation*}%
and (8) follows.

If $\sup_{s,\upsilon }|g(s,\upsilon )|<\infty $ $\mathbf{P}$-a.s., then by (%
\ref{eight}) 
\begin{equation*}
\mathbf{E}\big[A_{\tau }^{q}]\leq C\mathbf{E[}\sup_{s,\upsilon
}|g(s,\upsilon )|^{p-2}L_{\tau }+L_{\tau }^{q}\big]. 
\end{equation*}%
Applying Young's inequality, we obtain (\ref{nine}). The lemma is proved.
\end{proof}

\begin{remark}
\label{hr0}{Under assumptions of Lemma \ref{hl0}, the estimates (\ref{hf0})
imply that}%
\begin{equation*}
\mathbf{E}\int_{0}^{\tau }\int_U |g(s,\upsilon )|^{p}\Pi (d\upsilon )ds\leq C%
\mathbf{E}[A_{\tau }^{p/2}]\leq C\mathbf{E}\big[\sup_{t\leq \tau }Q_{t}^{p}%
\big]. 
\end{equation*}
\end{remark}

We will need the following estimate of stochastic integrals (cf. \cite{m},
Lemma 1).

\begin{lemma}
\label{hl1}Let $(\mu _{s})$ be a measurable family of $\sigma $-finite
measures on a measurable space $(A,\mathcal{A})$. Let $g$ be a $\mathcal{R}(%
\mathbb{F})\otimes \mathcal{A\otimes U}$-measurable function on $[0,T]\times
A\times U$ such that 
\begin{equation*}
\int_{0}^{T}\int_{U}\bigg(\int_{A}g(s,a,\upsilon )\mu _{s}(da)\bigg)^{2}\Pi
(d\upsilon )ds<\infty 
\end{equation*}%
$\mathbf{P}$-a.s. Then for each $p\geq 2$ there is a constant $C$ such that%
\begin{eqnarray*}
&&\bigg\vert\sup_{0\leq t\leq T}\bigg\vert\int_{0}^{t}\int_{U}\int_{A}g(t,a,%
\upsilon )\mu _{t}(da)q(dt,d\upsilon )\bigg\vert\,\bigg\vert_{p} \\
&&\qquad \leq C\sum_{l=2,p}\sup_{t,a}|g(t,a,\cdot )|_{l,p}\bigg(%
\int_{0}^{T}|\mu _{t}|(A)^{l}dt\bigg)^{1/l},
\end{eqnarray*}%
where $|\mu _{t}|$ is the variation of $\mu _{t}$.
\end{lemma}

\begin{proof}
By Lemma \ref{hl0}, we have%
\begin{eqnarray*}
&&\mathbf{E}\bigg[\sup_{0\leq t\leq T}\bigg\vert\int_{0}^{t}\int_{U}%
\int_{A}g(t,a,\upsilon )\mu _{t}(da)q(dt,d\upsilon )\bigg\vert^{p}\bigg] \\
&&\qquad \leq C\sum_{l=2,p}\mathbf{E}\bigg[\bigg(\int_{0}^{T}\int_{U}%
\bigg\vert\int_{A}g(t,a,\upsilon )\mu _{t}(da)\bigg\vert^{l}\Pi (d\upsilon
)dt\bigg)^{p/l}\bigg].
\end{eqnarray*}%
By generalized Minkowsky's inequality,%
\begin{eqnarray*}
&&\bigg\vert\bigg(\int_{0}^{T}\int_{U}\bigg\vert\int_{A}g(t,a,\upsilon )\mu
_{t}(da)\bigg\vert^{l}\Pi (d\upsilon )dt\bigg)^{1/l}\bigg\vert_{p} \\
&&\qquad \leq \bigg(\int_{0}^{T}\bigg\vert\int_{A}g(t,a,\cdot )\mu _{t}(da)%
\bigg\vert_{l,p}^{l}dt\bigg)^{1/l} \\
&&\qquad \leq \bigg(\int_{0}^{T}\bigg(~\int_{A}|g(t,a,\cdot )|_{l,p}|\mu
_{t}|(da)\bigg)^{l}dt\bigg)^{1/l} \\
&&\qquad \leq \sup_{t,a}|g(t,a,\cdot )|_{l,p}\bigg(\int_{0}^{T}|\mu
_{t}|(A)^{l}dt\bigg)^{1/l}.
\end{eqnarray*}%
The lemma is proved.
\end{proof}

\subsection{Solution for smooth input functions}

First we solve (\ref{eq1}) for smooth input functions $f,g.$

\begin{lemma}
\label{l0}Let $p\geq 2$, $f\in {C}_{p}^{\infty}(H),\ g\in {C}%
_{r,p}^{\infty}(H\times U),\ r=2,p.$

Then there is a unique $u\in {C}_{p}^{\infty}(H)$ solving (\ref{eq1}).
Moreover, $\mathbf{P}$-a.s. $u(t,x)$ is cadlag in $t$ and smooth in $x$.
Also, for each $(t,x)$ $\mathbf{P}$-a.s. 
\begin{equation*}
u(t,x)=R_{\lambda }f(t,x)+\widetilde{R}_{\lambda }g(t,x), 
\end{equation*}%
where%
\begin{eqnarray}
R_{\lambda }f(t,x) &=&\int_{0}^{t}\mathcal{F}^{-1}[K_{s,t}^{\lambda }(\xi )%
\mathcal{F}f(s,\xi )](x)ds,  \notag \\
&&  \label{rez} \\
\widetilde{R}_{\lambda }g(t,x) &=&\int_{0}^{t}\int_{U}\mathcal{F}%
^{-1}[K_{s,t}^{\lambda }(\xi )\mathcal{F}g(s,\xi ,\upsilon
)](x)q(ds,d\upsilon ),  \notag
\end{eqnarray}%
and%
\begin{equation*}
K_{s,t}^{\lambda }(\xi )=\exp \left\{ \int_{s}^{t}(\psi ^{(\alpha )}(r,\xi
)-\lambda )dr\right\} ,0\leq s\leq t\leq T. 
\end{equation*}
\end{lemma}

\begin{proof}
\emph{Existence.} We take a complete probability space $(\Omega ^{\prime },%
\mathcal{F}^{\prime },\mathbf{P}^{\prime })$ with a filtration of $\sigma $%
-algebras $\mathbb{F}^{\prime }=(\mathcal{F}_{t}^{\prime })$ satisfying
usual conditions and an adapted stable process $Z_{t}=Z_{t}(\omega ^{\prime
}),\omega ^{\prime }\in \Omega ^{\prime },t\in \lbrack 0,T],$ on it defined
by%
\begin{eqnarray}
Z_{t} &=&\int_{0}^{t}\sqrt{B(s)}dW_{s},\alpha =2,  \notag \\
Z_{t} &=&\int_{0}^{t}\int_{\mathbf{R}_{0}^{d}}yq^{Z}(ds,dy),~\alpha \in
(1,2),  \notag \\
Z_{t} &=&\int_{0}^{t}b(s)ds+\int_{0}^{t}\int_{|y|\leq
1}yq^{Z}(ds,dy)+\int_{0}^{t}\int_{|y|>1}yp^{Z}(ds,dy),\alpha =1,
\label{maf1} \\
Z_{t} &=&\int_{0}^{t}\int_{\mathbf{R}_{0}^{d}}yp^{Z}(ds,dy),~\alpha \in
(0,1),  \notag
\end{eqnarray}%
where $W_{t}$ is a standard Wiener process in $\mathbf{R}^{d}$, $%
p^{Z}(ds,dy)=p^{Z}(\omega ^{\prime },ds,dy)$ is a Poisson point process on $%
[0,T]\times \mathbf{R}_{0}^{d}$ and 
\begin{equation}
q^{Z}(ds,dy)=p^{Z}(ds,dy)-m^{(\alpha )}(s,y)\frac{dy}{|y|^{d+\alpha }}ds
\label{z1}
\end{equation}%
is a $(\mathbb{F}^{\prime },\mathbf{P}^{\prime })$-martingale measure.

Consider the product of probability spaces%
\begin{equation*}
(\widetilde{\Omega},\mathcal{\widetilde{F}}^{\prime },\mathbf{\widetilde{P}}%
)=(\Omega \times \Omega ^{\prime },\mathcal{F\otimes F}^{\prime },\mathbf{%
P\times P}^{\prime }). 
\end{equation*}%
We will denote $\mathbf{\widetilde{E}}$ the expectation with respect to $%
\mathbf{\widetilde{P}}$. Let $\mathcal{\widetilde{F}}$ be the completion of $%
\mathcal{\widetilde{F}}^{\prime }$ and $\mathbb{\widetilde{F}=(\mathcal{%
\widetilde{F}}}_{t})$ be the usual augmentation of $\mathbb{(}\mathcal{F}%
_{t}\otimes \mathcal{F}_{t}^{\prime })$ (see \cite{delmey}). Let $\mathbb{%
\widetilde{F}}^{Z}=\mathbb{(\mathcal{\widetilde{F}}}_{t}^{Z})$ be the usual
augmentation of $\left( \mathcal{F\otimes F}_{t}^{\prime }\right) $ and $%
\mathbb{\widetilde{F}}^{q}=(\mathcal{\widetilde{F}}_{t}^{q})$ be the usual
augmentation of $\left( \mathcal{F}_{t}\mathcal{\otimes F}^{\prime }\right) $%
. Obviously, $q^{Z}(ds,dy)$ is $(\mathbb{\widetilde{F}}^{Z},\mathbf{%
\widetilde{P}})$- and $(\mathbb{\widetilde{F}},\mathbf{\widetilde{P}})$%
-martingale measure. Note that $q(dt,d\upsilon )$ is $(\mathbb{\widetilde{F}}%
,\mathbf{\widetilde{P}})$- and $(\mathbb{\widetilde{F}}^{q},\mathbf{%
\widetilde{P}}$)-martingale measure as well. Denote for a measurable
function $F$ on $\widetilde{\Omega}$%
\begin{equation*}
\mathbf{\widetilde{E}}^{Z}F=\int_{\Omega ^{\prime }}F(\omega ,\omega
^{\prime })\mathbf{P}^{\prime }(d\omega ^{\prime }). 
\end{equation*}%
First, we claim that for each $\mathcal{P}(\mathbb{\widetilde{F}}^{q}\mathbb{%
)\otimes }\mathcal{U}$-measurable function $g(s,\upsilon )$ such that%
\begin{equation*}
\mathbf{\widetilde{E}}\int_{0}^{T}\int_{U}g(s,\upsilon )^{2}\Pi (d\upsilon
)ds<\infty , 
\end{equation*}%
we have $\mathbf{P}$\textbf{-}a.s. for all $t$ 
\begin{equation}
\mathbf{\widetilde{E}}^{Z}\int_{0}^{t}\int_{U}g(s,\upsilon )q(ds,d\upsilon
)=\int_{0}^{t}\int_{U}\mathbf{\widetilde{E}}^{Z}[g(s,\upsilon
)]q(ds,d\upsilon ).  \label{corf2}
\end{equation}%
It is enough to show that%
\begin{equation}
\mathbf{\widetilde{E}}^{Z}\int_{0}^{t}\int_{U}1_{U_{m}}(\upsilon
)g(s,\upsilon )q(ds,d\upsilon )=\int_{0}^{t}\int_{U}1_{U_{m}}(\upsilon )%
\mathbf{\widetilde{E}}^{Z}[g(s,\upsilon )]q(ds,d\upsilon )  \label{corf2'}
\end{equation}%
for all $m$ (we have $\cup _{m}U_{m}=U,\Pi (U_{m})<\infty $ for all $m)$.

Let $\mathcal{K}$ be the collection of all bounded functions $h$ on $%
\widetilde{\Omega}\times U$ of the form%
\begin{equation*}
h(s,\upsilon )=\sum_{k=1}^{n}\bar{g}_{k}(s,\upsilon )1_{A_{k}},n\geq 1, 
\end{equation*}%
where $A_{k}\in \mathcal{F}^{\prime },$ $\bar{g}_{k}$ are bounded
real-valued $\mathcal{P}(\mathbb{F})\otimes \mathcal{U}$-measurable
functions on $\Omega \times U$. Since $\mathcal{F}^{\prime }\subseteq 
\mathcal{\widetilde{F}}_{0}^{q}$, we have $\mathbf{\widetilde{P}}$-a.s. for
all $t$ and $m$ 
\begin{equation*}
\int_{0}^{t}\int_U 1_{U_{m}}(\upsilon )h(s,\upsilon )q(ds,d\upsilon
)=\sum_{k=1}^{n}\int_U1_{A_{k}}\int_{0}^{t}1_{U_{m}}(\upsilon )\bar{g}%
_{k}(s,\upsilon )q(ds,d\upsilon ) 
\end{equation*}%
and 
\begin{eqnarray*}
\mathbf{\widetilde{E}}^{Z}\int_{0}^{t}\int_U 1_{U_{m}}(\upsilon
)h(s,\upsilon )q(ds,d\upsilon ) &=&\sum_{k=1}^{n}\mathbf{P}^{\prime
}(A_{k})\int_{0}^{t}\int_U 1_{U_{m}}(\upsilon )\bar{g}_{k}(s,\upsilon
)q(ds,d\upsilon ) \\
&=&\int_{0}^{t}\int_U \sum_{k=1}^{n}1_{U_{m}}(\upsilon )\mathbf{P}^{\prime
}(A_{k})\bar{g}_{k}(s,\upsilon )q(ds,d\upsilon ) \\
&=&\int_{0}^{t}\int_U \mathbf{\widetilde{E}}^{Z}[h(s,\upsilon
)]1_{U_{m}}(\upsilon )q(ds,d\upsilon ).
\end{eqnarray*}%
Therefore (\ref{corf2'}) holds by the monotone class theorem (see Theorem
I-21 in \cite{delmey}), and (\ref{corf2}) follows as well.

Applying Lemma \ref{cl1} (see Appendix) in $(\widetilde{\Omega },\mathcal{%
\widetilde{F}},\mathbf{\widetilde{P}})$ with the filtration $\mathbb{%
\widetilde{F}=(\mathcal{\widetilde{F}}}_{t})$ for 
\begin{equation*}
\int_{0}^{t}\int_{U}e^{\lambda s}g(s,x-Z_{s},\upsilon )q(ds,dz),0\leq t\leq
T, 
\end{equation*}%
we find that there is a $\mathcal{P(\mathbb{\widetilde{F})\otimes }B}(%
\mathbf{R}^{d})$-measurable real-valued function $w(t,x)$ such that $\mathbf{%
\widetilde{P}}$-a.s. $w(t,x)$ is cadlag in $t$ and smooth in $x$. Moreover,
for each $x\in \mathbf{R}^{d},\gamma \in \mathbf{N}_{0}^{d},\mathbf{%
\widetilde{P}}$-a.s. for all $t$%
\begin{equation*}
\partial _{x}^{\gamma }w(t,x)=\int_{0}^{t}e^{\lambda s}\partial _{x}^{\gamma
}f(s,x-Z_{s})ds+\int_{0}^{t}\int_{U}e^{\lambda s}\partial _{x}^{\gamma
}g(s,x-Z_{s},\upsilon )q(ds,d\upsilon ). 
\end{equation*}%
Also, for each $\gamma \in \mathbf{N}_{0}^{d},R>0,$%
\begin{equation*}
\mathbf{\widetilde{E}}\Big[\sup_{t\leq T,|x|\leq R}|\partial _{x}^{\gamma
}w(t,x)|^{p}\Big]+\sup_{x}\mathbf{\widetilde{E}}\Big[\sup_{t\leq T}|\partial
_{x}^{\gamma }w(t,x)|^{p}\Big]<\infty . 
\end{equation*}

Let $\left( \Delta _{k}^{n}\right) _{k\geq 1}$ be a sequence of measurable
partitions of $\mathbf{R}^{d}$ such that every $\Delta _{k}^{n}$ has a
diameter \ smaller than $1/n$ and let $z_{k}^{n}\in \Delta _{k}^{n}.$ We fix 
$t,x$ and define $Z_{t}^{n}=z_{k}^{n}$ if $Z_{t}\in \Delta _{k}^{n},k\geq 1.$

Let $\widetilde{g}(s,x,\upsilon )$ be $\mathcal{P}(\mathbb{F})\otimes 
\mathcal{B}(\mathbf{R}^{d})\otimes \mathcal{U}$-measurable modification of $%
g(s,x,\upsilon )$ in Lemma \ref{cl1}. Since $\widetilde{g}(s,x+Z_{t}-Z_{s}),%
\widetilde{g}(s,x+Z_{t}^{n}-Z_{s})$ are $\mathcal{P}(\mathbb{F})\otimes 
\mathcal{F}^{\prime }\subseteq \mathcal{P}(\mathbb{\widetilde{F}}^{q})$%
-measurable, for each $t,x,\gamma \in \mathbf{N}_{0}^{d}$ $\mathbf{%
\widetilde{P}}$-a.s. 
\begin{eqnarray}
\partial _{x}^{\gamma }w(t,x+Z_{t}^{n}) &=&\sum_{k}\partial _{x}^{\gamma
}w(t,x+z_{k}^{n})1_{\Delta _{k}^{n}}(Z_{t})  \notag \\
&=&\int_{0}^{t}e^{\lambda s}\partial _{x}^{\gamma }f(s,x+Z_{t}^{n}-Z_{s})ds+
\notag \\
&&+\sum_{k}1_{\Delta _{k}^{n}}(Z_{t})\int_{0}^{t}\int_{U}e^{\lambda
s}\partial _{x}^{\gamma }g(s,x+z_{k}^{n}-Z_{s},\upsilon )q(ds,d\upsilon )
\label{crf4} \\
&=&\int_{0}^{t}e^{\lambda s}\partial _{x}^{\gamma }f(s,x+Z_{t}^{n}-Z_{s})ds 
\notag \\
&&+\int_{0}^{t}\int_{U}e^{\lambda s}\partial _{x}^{\gamma }\widetilde{g}%
(s,x+Z_{t}^{n}-Z_{s},\upsilon )q(ds,d\upsilon )  \notag
\end{eqnarray}%
and, by the Minkowsky inequality and Lemma \ref{hl0},%
\begin{eqnarray}
&&{\widetilde{\mathbf{E}}}|\partial _{x}^{\gamma }w(t,x+Z_{t}^{n})|^{p}\leq C%
\widetilde{\mathbf{E}}\bigg\{\bigg\vert\int_{0}^{t}e^{\lambda s}\partial
_{x}^{\gamma }{f}(s,x+Z_{t}^{n}-Z_{s-})ds\bigg\vert^{p}+  \label{crf5} \\
&&\qquad +\bigg\vert\int_{0}^{t}\int_{U}e^{\lambda s}\partial _{x}^{\gamma }%
\widetilde{g}(s,x+Z_{t}^{n}-Z_{s-},\upsilon )q(ds,d\upsilon )\bigg\vert^{p}%
\bigg\}  \notag \\
&&\quad \leq C\widetilde{\mathbf{E}}\bigg\{\int_{0}^{t}e^{p\lambda s}%
\big\vert\partial _{x}^{\gamma }{f}(s,x+Z_{t}^{n}-Z_{s-})\big\vert^{p}ds+ 
\notag \\
&&\qquad +\sum_{l=2,p}\bigg(\int_{0}^{t}\int_{U}e^{l\lambda s}\big\vert%
\partial _{x}^{\gamma }\widetilde{g}(s,x+Z_{t}^{n}-Z_{s-},\upsilon )\big\vert%
^{l}\Pi (d\upsilon )ds\bigg)^{p/l}\bigg\}  \notag \\
&&\quad \leq Ce^{p\lambda t}\sup_{s,y}\biggl[\mathbf{E}\big\vert\partial
_{y}^{\gamma }f(s,y)\big\vert^{p}+\sum_{l=2,p}\mathbf{E}\biggl(\int_{U}%
\big\vert\partial _{x}^{\gamma }g(s,y,\upsilon )\big\vert^{l}\Pi (d\upsilon )%
\biggr)^{p/l}\biggr],  \notag
\end{eqnarray}%
where the constant $C$ does not depend on $n$.

Therefore, passing to the limit in (\ref{crf4}) as $n\rightarrow \infty $
(the stochastic integral is regarded as an integral of $\mathcal{P}(\mathbb{%
\widetilde{F}}^{q})$ function), we obtain that for each $(t,x),\gamma \in 
\mathbf{N}_{0}^{d},\mathbf{\widetilde{P}}$-a.s. 
\begin{eqnarray}
\partial _{x}^{\gamma }w(t,x+Z_{t})&=&\int_{0}^{t}e^{\lambda s}\partial
_{x}^{\gamma }f(s,x+Z_{t}-Z_{s})ds  \notag \\
&&+\int_{0}^{t}\int_{U}e^{\lambda s}\partial _{x}^{\gamma }\widetilde{g}%
(s,x+Z_{t}-Z_{s},\upsilon )q(ds,d\upsilon).  \label{crf6}
\end{eqnarray}%
By (\ref{crf5}) and the Fatou lemma, we have for all $\gamma \in \mathbf{N}%
_{0}^{d}$ 
\begin{eqnarray}
\sup_{t,x}\mathbf{\widetilde{E}}\big\vert \partial _{x}^{\gamma }w(t,x+Z_{t})%
\big\vert ^{p}&\leq & Ce^{p\lambda T}\sup_{s,y}\bigg[\mathbf{E}\big\vert %
\partial _{y}^{\gamma }f(s,y)\big\vert ^{p}  \notag \\
&&+\sum_{l=2,p}\mathbf{E}\bigg(\int_{U}\big\vert \partial _{x}^{\gamma
}g(s,y,\upsilon )\big\vert ^{l}\Pi (d\upsilon )\bigg)^{p/l}\bigg].
\label{crf7}
\end{eqnarray}%
Therefore, for each for each $R>0,\gamma \in \mathbf{N}_{0}^{d},$ 
\begin{equation*}
\sup_{t}\mathbf{\widetilde{E}}\int_{|x|<R}\big\vert \partial _{x}^{\gamma
}w(t,x+Z_{t})\big\vert ^{p}dx<\infty , 
\end{equation*}%
and, by the Sobolev embedding theorem$,$%
\begin{equation*}
\sup_{t}\mathbf{\widetilde{E}}\Big[ \sup_{|x|\leq R}\big\vert \partial
_{x}^{\gamma }w(t,x+Z_{t})\big\vert ^{p}\Big]<\infty . 
\end{equation*}

Therefore for each $t$, $\mathbf{P}$-a.s. the function $\widetilde{u}(t,x)=%
\mathbf{\widetilde{E}}^{Z}w(t,x+Z_{t})$ is smooth in $x$ and according to (%
\ref{corf2}), for each $(t,x),\gamma \in \mathbf{N}_{0}^{d},$ $\mathbf{P}$%
-a.s. 
\begin{eqnarray}
\partial _{x}^{\gamma }\mathbf{\widetilde{E}}^{Z}w(t,x+Z_{t}) &=&\mathbf{%
\widetilde{E}}^{Z}\partial _{x}^{\gamma }w(t,x+Z_{t})  \notag \\
&=&\int_{0}^{t}e^{\lambda s}\mathbf{\widetilde{E}}^{Z}\partial _{x}^{\gamma
}f(s,x+Z_{t}-Z_{s-})ds  \label{crf9} \\
&&+\int_{0}^{t}\int_{U}e^{\lambda s}\mathbf{\widetilde{E}}^{Z}[\partial
_{x}^{\gamma }g(s,x+Z_{t}-Z_{s-},\upsilon )]q(ds,d\upsilon ).  \notag
\end{eqnarray}%
In addition, by Lemma \ref{cl1} and (\ref{crf9}), for each $R>0,\gamma \in 
\mathbf{N}_{0}^{d},$ 
\begin{equation}
\sup_{t,x}\mathbf{E}|\widetilde{u}(t,x)|^{p}+\sup_{t}\mathbf{E}%
\sup_{|x|<R}|\partial _{x}^{\gamma }\widetilde{u}(t,x)|^{p}<\infty .
\label{crf10}
\end{equation}

On the other hand, by the It\^{o}-Wentzell formula (see e.g. \cite{mik3}),
for each $x$ $\mathbf{\widetilde{P}}$-a.s. for all $\gamma \in \mathbf{N}%
_{0}^{d},t\in \lbrack 0,T],$%
\begin{eqnarray*}
e^{-\lambda t}\partial _{x}^{\gamma }w(t,x &+&Z_{t})= \\
&=&\int_{0}^{t}\partial _{x}^{\gamma }f(s,x)ds+\int_{0}^{t}\int_{U}\partial
_{x}^{\gamma }g(s,x,\upsilon )q(ds,d\upsilon ) \\
&&+\int_{0}^{t}e^{-\lambda s}\big(\nabla \partial _{x}^{\gamma
}w(s,x+Z_{s}),b(s)\big)1_{\alpha =1}ds \\
&&+\int_{0}^{t}\int e^{-\lambda s}[\partial _{x}^{\gamma
}w(s,x+Z_{s-}+y)-\partial _{x}^{\gamma }w(s,x+Z_{s-}) \\
&&-1_{\alpha \geq 1}\big(\nabla \partial _{x}^{\gamma }w(s,x+Z_{s-}),y\big)%
1_{|y|\leq 1}]p^{Z}(ds,dy) \\
&&+\int_{0}^{t}\int_{|y|\leq 1}e^{-\lambda s}1_{\alpha \geq 1}\big(\nabla
\partial _{x}^{\gamma }w(s,x+Z_{s-}),y\big)q^{Z}(ds,dy) \\
&&-\int_{0}^{t}\lambda e^{-\lambda s}\partial _{x}^{\gamma }w(s,x+Z_{s})ds.
\end{eqnarray*}%
By (\ref{corf2}) and (\ref{crf10}), for each $(t,x)$ we have $\mathbf{P}$%
-a.s. for all $\gamma \in \mathbf{N}_{0}^{d},$%
\begin{eqnarray*}
\partial _{x}^{\gamma }u(t,x) &=&e^{-\lambda t}\partial _{x}^{\gamma }%
\widetilde{u}(t,x)=\int_{0}^{t}\partial _{x}^{\gamma
}f(s,x)ds+\int_{0}^{t}\int_{U}\partial _{x}^{\gamma }g(s,x,\upsilon
)q(ds,d\upsilon ) \\
&&+\int_{0}^{t}\bigg(\nabla \partial _{x}^{\gamma }u(s,x),\bigg[%
b(s)1_{\alpha =1}-\int_{|y|>1}1_{\alpha >1}ym^{(\alpha )}(s,y)\frac{dy}{%
|y|^{d+\alpha }}\bigg]\bigg)ds \\
&&-\int_{0}^{t}\lambda \partial _{x}^{\gamma }u(s,x)ds+\int_{0}^{t}\int
[\partial _{x}^{\gamma }u(s,x+y)-\partial _{x}^{\gamma }u(s,x) \\
&&-1_{\alpha \geq 1}\big(\nabla \partial _{x}^{\gamma }u(s,x),y\big)%
1_{|y|\leq 1}]m^{(\alpha )}(s,y)\frac{dyds}{|y|^{d+\alpha }}.
\end{eqnarray*}%
According to Lemma \ref{cl1} in Appendix, the right-hand side of this
equation has $\mathbf{P}$-a.s. cadlag in $t$ and smooth in $x$ modification.

Finally, note that (\ref{crf9}) implies (\ref{rez}).

\emph{Uniqueness. }Assume $f=0,g=0$ and $u\in {C}^{\infty }(H)$ is a
deterministic solution of (\ref{eq1}). We fix $(t_{0},x)$ and show that $%
u(t_{0},x)=0.$ Let $Y_{t},t\in \lbrack 0,t_{0}]$ be a stable process defined
on some probability space by the same formulas (\ref{maf1}) as the process $%
Z $ with $p^{Z}$ and $q^{Z}$ replaced by $p^{Y}$ and $q^{Y}$, where 
\begin{equation*}
q^{Y}(ds,dy)=p^{Y}(ds,dy)-m^{(\alpha )}(t_{0}-s,y)\frac{dy}{|y|^{d+\alpha }}%
ds 
\end{equation*}%
is a martingale measure, $p^{Y}(ds,dy)$ is a Poisson point process on $%
[0,t_{0}]\times \mathbf{R}_{0}^{d}.$ By It\^{o} formula, 
\begin{eqnarray*}
-u(t_{0},x) &=&e^{-\lambda t_{0}}u(0,x+Y_{t_{0}})-u(t_{0},x)= \\
&=&\int_{0}^{t_{0}}e^{-\lambda t}\Bigl(\frac{\partial u}{\partial t}%
+A^{(\alpha )}u-\lambda u\Bigr)(t_{0}-t,x+Y_{t})dt=0.
\end{eqnarray*}%
The lemma is proved.
\end{proof}

The uniqueness for (\ref{eq1}) holds in ${C}^{\alpha ,\beta}_p(H)$ as well.

\begin{corollary}
\label{lu}There is at most one solution $u\in {C}^{\alpha ,\beta}_p(H)$ of (%
\ref{eq1}).
\end{corollary}

\begin{proof}
Let $u\in C^{\alpha ,\beta }(H)$ be a deterministic solution of (\ref{eq1}\.{%
)}, $\zeta \in C_{0}^{\infty }(\mathbf{R}^{d}),\varepsilon >0,\zeta
_{\varepsilon }(x)=\varepsilon ^{-d}\zeta (x/\varepsilon )$ and 
\begin{equation*}
u_{\varepsilon }(t,x)=\int u(t,y)\zeta _{\varepsilon }(x-y)dy. 
\end{equation*}

Taking the convolution of the both sides of (\ref{eq1}) with $%
\zeta_{\varepsilon }$, we find that $u_{\varepsilon }\in C^{\infty }(\mathbf{%
R}^{d})$ solves (\ref{eq1})$.$ Therefore $u_{\varepsilon }(t,x)=0$ for all $%
\varepsilon >0$ and $u(t,x)=0.$ So, the statement follows.
\end{proof}

Let $Z_{t},\ t\in \lbrack 0,T]$, be the random process defined by (\ref{maf1}%
). Notice that the function $K_{s,t}^{0}(\xi )$ is the characteristic
function of the increment $Z_{t}-Z_{s}$, $0\leqslant s\leqslant t\leqslant T$%
. According to assumption~A, $\int |K_{s,t}^{0}(\xi )|d\xi <\infty .$
Therefore the function 
\begin{equation*}
G_{s,t}(x)=\mathcal{F}^{-1}\bigl\{K_{s,t}^{0}(\xi )\bigr\}(x),\quad
0\leqslant s\leqslant t\leqslant T, 
\end{equation*}%
is the probability density of the increment $Z_{t}-Z_{s}$. Hence, 
\begin{equation}
G_{s,t}(x)\geqslant 0,\quad \int G_{s,t}(x)dx=1.  \label{22}
\end{equation}%
So, if assumption~A is satisfied, then for $f\in C_{p}^{\infty }(H)$ and $%
g\in C_{r,p}^{\infty }(H\times U),\ r=2,p,$ 
\begin{eqnarray}
R_{\lambda }f(t,x) &=&\int_{0}^{t}\bigl[G_{s,t}^{\lambda }(\cdot )\ast
f(s,\cdot )\bigr](x)ds,  \label{23} \\
\widetilde{R}_{\lambda }g(t,x) &=&\int_{0}^{t}\int \bigl[G_{s,t}^{\lambda
}(\cdot )\ast g(s,\cdot ,\upsilon )\bigr](x)q(ds,d\upsilon ),  \label{24}
\end{eqnarray}%
where 
\begin{equation*}
G_{s,t}^{\lambda }(x)=e^{-\lambda (t-s)}G_{s,t}(x) 
\end{equation*}%
and $\ast $ denotes the convolution.

It is easy to derive that $R_{\lambda }f,\ \widetilde{R}_{\lambda }g\in
B_{p}(H)$ if $f\in B_{p}(H)$ and $g\in B_{r,p}(H\times U)$, $r=2,p,\
p\geqslant 2.$ Indeed, by Lemma \ref{hl1} and (\ref{22}), (\ref{24}), 
\begin{eqnarray}
|\widetilde{R}_{\lambda }g(t,x)|_{p} &\leqslant
&C\sum_{r=2,p}\sup_{s\leqslant t,y}|g(s,y,\cdot )|_{r,p}\biggl\{\int_{0}^{t}%
\biggl(\int G_{s,t}^{\lambda }(y)dy\biggr)^{r}ds\biggr\}^{1/r}  \notag \\
&\leqslant &C\sum_{r=2,p}\Bigl(t\wedge \frac{1}{\lambda }\Bigr)^{1/r}\Vert
g\Vert _{r,p}.  \label{25}
\end{eqnarray}%
By the Minkowsky inequality and (\ref{22}), (\ref{23}), 
\begin{eqnarray}
|R_{\lambda }f(t,x)|_{p} &\leqslant &\int_{0}^{t}\int G_{s,t}^{\lambda
}(x-y)|f(s,y)|_{p}dy\,ds\leqslant  \notag \\
&\leqslant &\sup_{s\leqslant t,y}|f(s,y)|_{p}\int_{0}^{t}e^{-\lambda
(t-s)}ds\leqslant  \notag \\
&\leqslant &\Bigl(t\wedge \frac{1}{\lambda }\Bigr)\Vert f\Vert _{p}.
\label{26}
\end{eqnarray}

\subsection{Characterization of stochastic H\"{o}lder spaces}

For a characterization of our function spaces we will use the following
construction (see [1]). Let $\mathcal{S}(\mathbf{R}^{d})$ be the Schwartz
space of all real-valued rapidly decreasing infinitely differentiable
functions on $\mathbf{R}^{d}$. By Lemma 6.1.7 in [1], there exist a function 
$\phi \in C_{0}^{\infty }(\mathbf{R}^{d})$ such that $\mathrm{supp}\,\phi
=\{\xi :\frac{1}{2}\leqslant |\xi |\leqslant 2\}$, $\phi (\xi )>0$ if $%
2^{-1}<|\xi |<2$ and 
\begin{equation}
\sum_{j=-\infty }^{\infty }\phi (2^{-j}\xi )=1\quad \text{if }\xi \neq 0.
\label{27}
\end{equation}%
Define the functions $\varphi _{k}\in \mathcal{S}(\mathbf{R}^{d}),$ $k=0,\pm
1,\ldots ,$ by 
\begin{equation}
\mathcal{F}\varphi _{k}(\xi )=\phi (2^{-k}\xi ),  \label{28}
\end{equation}%
and $\psi \in \mathcal{S}(\mathbf{R}^{d})$ by 
\begin{equation}
\mathcal{F}\psi (\xi )=1-\sum_{k\geqslant 1}\mathcal{F}\varphi _{k}(\xi ).
\label{29}
\end{equation}

The following results are proved in \cite{mip2}. We simply take $%
V=L_{p}(\Omega ,\mathcal{F},\mathbf{P})$ or $V=L_{p}(\Omega ,\mathcal{F},%
\mathbf{P};L_{r}(U,\mathcal{U},\Pi ))$, $r>1$, $p>1$, in Lemma 12 and
Corollary 13 of \cite{mip2}.

\begin{lemma}
\label{lem4} Let $\alpha \in[0,2)$, $\beta \in(0,1)$, $r>1$ and $p>1.$ Then
the following statements hold:

$\mathrm{(i)}$ the norm $\Vert u\Vert_{\alpha ,\beta ;p} $ is equivalent to
the norm 
\begin{equation*}
\Vert \psi\ast u\Vert_p +\sup_{j\geq1} 2^{(\alpha +\beta )j} \Vert \varphi_j
\ast u\Vert _p; 
\end{equation*}

$\mathrm{(ii)}$ the norm $\Vert g\Vert_{\alpha ,\beta ;r,p} $ is equivalent
to the norm 
\begin{equation*}
\Vert \psi\ast g\Vert_{r,p} +\sup_{j\geq1} 2^{(\alpha +\beta )j} \Vert
\varphi_j \ast g\Vert _{r,p} . 
\end{equation*}
\end{lemma}

\begin{lemma}
\label{lem5} Let $\alpha \in (0,2)$, $\beta \in (0,1)$, $r>1$, $p>1$, $u\in
C_{p}^{\alpha ,\beta }(H)$, $g\in C_{r,p}^{\alpha ,\beta }(H\times U)$, and 
\begin{eqnarray*}
u_{n} &=&{}\psi \ast u+\sum_{j=1}^{n}\varphi _{j}\ast u, \\
g_{n} &=&{}\psi \ast u+\sum_{j=1}^{n}\varphi _{j}\ast g,\quad n\geq 1.
\end{eqnarray*}

Then $u_n\in C^{\infty}_p(H)$, $g_n\in C^{\infty}_{r,p}(H\times U),$ 
\begin{eqnarray*}
\Vert u_n\Vert_{\alpha ,\beta ;p} &\leq& 2\Vert u\Vert_{\alpha ,\beta
;p},\quad \Vert g_n\Vert_{\alpha ,\beta ;r,p} \leq 2\Vert g\Vert_{\alpha
,\beta ;r,p}, \\
\Vert u\Vert_{\alpha ,\beta ;p} &\leq & \mathop{\mathrm{lim\,inf}}%
_{n\to\infty} \Vert u_n\Vert_{\alpha ,\beta ;p},\quad \Vert g\Vert_{\alpha
,\beta ;r,p} \leq \mathop{\mathrm{lim\,inf}}_{n\to\infty} \Vert
g_n\Vert_{\alpha ,\beta ;r,p}
\end{eqnarray*}
and for each $\beta ^{\prime }\in(0,\beta )$ 
\begin{equation*}
\Vert u_n-u\Vert_{\alpha ,\beta ^{\prime };p} \to 0,\quad \Vert
g_n-g\Vert_{\alpha ,\beta ^{\prime };r,p} \to 0,\ n\to\infty . 
\end{equation*}
\end{lemma}

\section{Proof of main result}

\subsection{Estimates of $R_{\protect\lambda }f$ and $\protect\widetilde R_{%
\protect\lambda }g$}

In order to prove Theorem \ref{main}, we need some estimates of the
functions $R_{\lambda }f$ and $\widetilde R_{\lambda }g$.

Let $\psi ,\varphi _{j},\ j\geqslant 0,$ be the functions defined by (\ref%
{28}), (\ref{29}) and 
\begin{eqnarray*}
f_{0}(t,\cdot ) &=&f(t,\cdot )\ast \psi ,\quad g_{0}(t.\cdot )=g(t,\cdot
)\ast \psi , \\
f_{j}(t,\cdot ) &=&f(t,\cdot )\ast \varphi _{j},\quad g_{j}(t.\cdot
)=g(t,\cdot )\ast \varphi _{j},\quad j\geqslant 1.
\end{eqnarray*}%
Obviously, 
\begin{eqnarray*}
{}\psi \ast R_{\lambda }f(t,\cdot ) &=&R_{\lambda }f_{0}(t,\cdot ),\quad
{}\psi \ast \widetilde{R}_{\lambda }g(t,\cdot )=\widetilde{R}_{\lambda
}g_{0}(t,\cdot ), \\
{}\varphi _{j}\ast R_{\lambda }f(t,\cdot ) &=&R_{\lambda }f_{j}(t,\cdot
),\quad \varphi _{j}\ast \widetilde{R}_{\lambda }g(t,\cdot )=\widetilde{R}%
_{\lambda }g_{j}(t,\cdot ),\quad j\geqslant 1.
\end{eqnarray*}

Let us introduce the functions 
\begin{eqnarray*}
\widetilde\varphi_0 &=& {}\psi + \varphi_0, \\
\widetilde \varphi_j &=& {}\varphi_{j-1} +\varphi_j +\varphi_{j+1},\quad
j\geqslant1 .
\end{eqnarray*}
Since 
\begin{equation*}
\psi = \psi\ast\widetilde\varphi_0,\quad \varphi_j = \varphi_j\ast\widetilde
\varphi_j, 
\end{equation*}
we have 
\begin{eqnarray*}
R_{\lambda }f_j(t,x) &=& \int_0^t \Bigl[ h^{\lambda ,j}_{s,t}(\cdot)\ast
f_j(s,\cdot)\Bigr](x)ds, \\
\widetilde R_{\lambda }g_j(t,x) &=& \int_0^t\int_U \Bigl[ h^{\lambda
,j}_{s,t}(\cdot)\ast g_j(s,\cdot,\upsilon )\Bigr](x)q(ds,d\upsilon ),
\end{eqnarray*}
where 
\begin{equation*}
h^{\lambda ,j}_{s,t}(x) = \mathcal{F}^{-1}\Bigl[ K^{\lambda }_{s,t}(\cdot)%
\mathcal{F}\widetilde\varphi_j\Bigr] (x),\quad j\geqslant0. 
\end{equation*}

According to Remark 10 \cite{mip2}, there is a finite family $\{\Gamma
_{j},j=1,\ldots ,N\}$ of open connected sets which covers the unit sphere $%
S^{d-1}$ and a family $\{\mathcal{O}_{j}(\overline{w}),\overline{w}\in
\Gamma _{j},j=1,\ldots ,N\}$ of infinitely differentiable orthogonal
transforms with the following properties: for each $w\in S^{d-1},\ \overline{%
w}=\Gamma _{j},\ j=1,\ldots ,N,$ and multiindex $\gamma $ 
\begin{equation*}
(\mathcal{O}_{j}(\overline{w})w,\overline{w})=w_{1},\quad |\partial ^{\gamma
}\mathcal{O}_{j}(\overline{w})|\leqslant C(\gamma ) 
\end{equation*}%
and for each $\xi /|\xi |\in \Gamma _{j},$ $j=1,\ldots ,N$ 
\begin{eqnarray*}
{}\psi ^{(\alpha )}(t,\xi ) &=&-|\xi |^{\alpha
}\int_{S^{d-1}}|w_{1}|^{\alpha }\Big[1-i\tan \frac{\alpha \pi }{2}%
\mathop{{\mathrm{sgn}}}w_{1}\,1_{\alpha \neq 1}- \\
&&-\frac{2}{\pi }\mathop{{\mathrm{sgn}}}w_{1}\ln |w_{1}|1_{\alpha =1}\Big]%
m^{(\alpha )}\Bigl(t,\mathcal{O}_{j}\Bigl(\frac{\xi }{|\xi |}\Bigr)w\Bigr)%
\mu _{d-1}(dw)+ \\
&&+i(b(t),\xi )1_{\alpha =1}+\frac{1}{2}\sum_{i,j=1}^{d}B^{ij}(t)\xi _{i}\xi
_{j}1_{\alpha =2}.
\end{eqnarray*}

Using these properties and assumptions~\textbf{A1, A2}, it is easy to derive
the following estimates:

(i) there is a constant $C=C(\alpha ,\mu ,C^{(\alpha )},d)$ such that for $%
0\leqslant s\leqslant t\leqslant T$ and multiindices $\gamma ,|\gamma
|\leqslant d_{0}$ 
\begin{equation}
\big\vert\partial _{\xi }^{\gamma }K_{s,t}^{\lambda }(\xi )\big\vert%
\leqslant Ce^{-\mu (t-s)(|\xi |^{\alpha }+\lambda )}\sum_{k\leqslant |\gamma
|}|\xi |^{k\alpha -|\gamma |}(t-s)^{k};  \label{30}
\end{equation}

(ii) for each $\kappa \in (0,1)$ there is a constant $C=C(\alpha ,\mu
,C^{(\alpha )},d,\kappa )$ such that for $0\leqslant s\leqslant t\leqslant
t^{\prime }\leqslant T$ and multiindices $\gamma ,|\gamma |\leqslant d_{0}$ 
\begin{eqnarray}
&&\big\vert\partial _{\xi }^{\gamma }\bigl[K_{s,t^{\prime }}^{\lambda }(\xi
)-K_{s,t}^{\lambda }(\xi )\bigr]\big\vert\leqslant  \label{31} \\
&\leqslant &Ce^{-\mu (t-s)(|\xi |^{\alpha }+\lambda )}\bigl[(t^{\prime
}-t)|\xi |^{\alpha }\bigr]^{\kappa }\sum_{k\leqslant |\gamma |}|\xi
|^{k\alpha -|\gamma |}(t-s)^{k}.  \notag
\end{eqnarray}

\begin{lemma}
\label{lem6} Let assumptions~\textbf{A1, A2} be satisfied and $0\leq s\leq
t\leq t^{\prime }\leq T$. Then:

$\mathrm{(i)}$ for all $j\geq 1$ 
\begin{eqnarray*}
\int \big\vert h_{s,t}^{\lambda ,j}(x)\big\vert dx &\leq &Ce^{-c(2^{j\alpha
}+\lambda )(t-s)}\sum_{k\leq d_{0}}\big[2^{j\alpha }(t-s)\big]^{k}, \\
\int \big\vert h_{s,t}^{\lambda ,0}(x)\big\vert dx &\leq &Ce^{-\lambda (t-s)}
\end{eqnarray*}%
where the constants $C=C(\alpha ,\mu ,d,C^{(\alpha )})$, $c=c(\alpha ,\mu
)>0.$

$\mathrm{(ii)}$ for each $\kappa \in (0,1)$ there is a constant $C=C(\kappa
,\alpha ,\mu ,d,C^{(\alpha )})$ such that for all $j\geq 1$ 
\begin{eqnarray*}
\int \big\vert h_{s,t^{\prime }}^{\lambda ,j}(x) &-&h_{s,t}^{\lambda ,j}(x)%
\big\vert dx\leq \\
&\leq &C[2^{j\alpha }(t^{\prime }-t)]^{\kappa }e^{-c(2^{j\alpha }+\lambda
)(t-s)}\sum_{k\leq d_{0}}\big[2^{j\alpha }(t-s)\big]^{k},
\end{eqnarray*}%
where the constant $c=c(\alpha ,\mu )>0$.
\end{lemma}

\begin{proof}
By Lemma 16 \cite{mip2}, 
\begin{equation*}
\int \big\vert h_{s,t}^{\lambda ,0}(x)\big\vert dx\leq Ce^{-\lambda (t-s)}. 
\end{equation*}%
According to (\ref{30}), (\ref{31}) and the definition of $\widetilde{%
\varphi }_{j}$, for $\kappa \in (0,1)$, multiindices $\gamma ,|\gamma |\leq
d_{0}$, and $j\geq 1$ 
\begin{eqnarray*}
\big\vert\partial _{\xi }^{\gamma }\bigl[K_{s,t}^{\lambda }(\xi )\mathcal{F}%
\widetilde{\varphi }_{j}(\xi )|\bigr]\big\vert &\leq &\sum_{\gamma ^{\prime
}+\gamma ^{\prime \prime }=\gamma }\big\vert\partial _{\xi }^{\gamma
^{\prime }}K_{s,t}^{\lambda }(\xi )\big\vert\,|\partial _{\xi }^{\gamma
^{\prime \prime }}\mathcal{F}\widetilde{\varphi }_{j}(\xi )\big\vert\leq \\
&\leq &C2^{-j|\gamma |}e^{-c(2^{j\alpha }+\lambda )(t-s)}\sum_{k\leq |\gamma
|}[2^{j\alpha }(t-s)]^{k}
\end{eqnarray*}%
and 
\begin{eqnarray*}
\Big\vert\partial _{\xi }^{\gamma }\bigl[\big(K_{s,t^{\prime }}^{\lambda
}(\xi )-K_{s,t}^{\lambda }(\xi )\bigr)\mathcal{F}\widetilde{\varphi }%
_{j}(\xi )\bigr]\Big\vert &\leq &\sum_{\gamma ^{\prime }+\gamma ^{\prime
\prime }=\gamma }\big\vert\partial _{\xi }^{\gamma ^{\prime }}\bigl(%
K_{s,t^{\prime }}^{\lambda }(\xi )-K_{s,t}^{\lambda }(\xi )\bigr)\big\vert\,%
\big\vert\partial _{\xi }^{\gamma ^{\prime \prime }}\mathcal{F}\widetilde{%
\varphi }_{j}(\xi )\big\vert\leq \\
&\leq &C2^{-j|\gamma |}\bigl[2^{j\alpha }(t^{\prime }-t)\bigr]^{\kappa
}e^{-c(2^{j\alpha }+\lambda )(t-s)}\sum_{k\leq |\gamma |}[2^{j\alpha
}(t-s)]^{k}.
\end{eqnarray*}%
These estimates, together with Parseval's equality, imply 
\begin{eqnarray*}
\int \big\vert(ix)^{\gamma }h_{s,t}^{\lambda ,j}(x)2^{j|\gamma |}\big\vert%
^{2}dx &=&\int \big\vert\partial _{\xi }^{\gamma }\bigl[K_{s,t}^{\lambda
}(\xi )\mathcal{F}\widetilde{\varphi }_{j}(\xi )\bigr]2^{j|\gamma |}\big\vert%
^{2}d\xi \leq \\
&\leq &C2^{jd}e^{-c(2^{j\alpha }+\lambda )(t-s)}\sum_{k\leq |\gamma
|}[2^{j\alpha }(t-s)]^{2k}
\end{eqnarray*}%
and 
\begin{eqnarray*}
&&\int \big\vert(ix)^{\gamma }\bigl[h_{s,t^{\prime }}^{\lambda
,j}(x)-h_{s,t}^{\lambda ,j}(x)\bigr]2^{j|\gamma |}\big\vert^{2}dx= \\
&&\qquad =\int \big\vert\partial _{\xi }^{\gamma }\bigl[K_{s,t^{\prime
}}^{\lambda }(\xi )-K_{s,t}^{\lambda }(\xi )\bigr]2^{j|\gamma |}\big\vert%
^{2}d\xi \leq \\
&&\qquad \leq C2^{jd}[2^{j\alpha }(t^{\prime }-t)]^{2\kappa
.}e^{-c(2^{j\alpha }+\lambda )(t-s)}\sum_{k\leq |\gamma |}[2^{j\alpha
}(t-s)]^{2k}.
\end{eqnarray*}%
Therefore, 
\begin{eqnarray*}
\int \big\vert h_{s,t}^{\lambda ,j}(x)\big\vert dx &\leq &\biggl(\int \frac{%
dx}{(1+|2^{j}x|)^{2d_{0}}}\biggr)^{1/2}\biggl(\int \bigl(1+|2^{j}x|\bigr)%
^{2d_{0}}\big\vert h_{s,t}^{\lambda ,j}(x)\big\vert^{2}dx\biggr)^{1/2}\leq \\
&\leq &Ce^{-c(2^{j\alpha }+\lambda )(t-s)}\sum_{k\leq d_{0}}[2^{j\alpha
}(t-s)]^{k}
\end{eqnarray*}%
and 
\begin{eqnarray*}
&&\int \big\vert h_{s,t^{\prime }}^{\lambda ,j}(x)-h_{s,t}^{\lambda ,j}(x)%
\big\vert dx\leq \\
&&\qquad \leq \biggl(\int \frac{dx}{(1+|2^{j}x|)^{2d_{0}}}\biggr)^{1/2}%
\biggl(\int \bigl(1+|2^{j}x|\bigr)^{2d_{0}}\big\vert h_{s,t^{\prime
}}^{\lambda ,j}(x)-h_{s,t}^{\lambda ,j}(x)\big\vert^{2}dx\biggr)^{1/2}\leq \\
&&\qquad \leq C\bigl[2^{j\alpha }(t^{\prime }-t)\bigr]^{\kappa
}e^{-c(2^{j\alpha }+\lambda )(t-s)}\sum_{k\leq d_{0}}[2^{j\alpha }(t-s)]^{k}.
\end{eqnarray*}%
The lemma is proved.
\end{proof}

\begin{lemma}
\label{lem7} Let $\alpha ,\alpha ^{\prime }\in (0,2]$, $\beta ,\beta
^{\prime }\in (0,1)$ and $p\geq 2$. Let assumptions~\textbf{A1, A2} be
satisfied, $\alpha (1-\frac{1}{p})+\beta =\alpha ^{\prime }+\beta ^{\prime }$
and $g\in C_{r,p}^{\alpha ^{\prime },\beta ^{\prime }}(H\times U)$, $r=2,p$.

Then there is a constant $C$ depending only on $\alpha ,\beta ,p,d,\mu,
C^{(\alpha )},T$ such that the following estimates hold:

$\mathrm{(i)}$ 
\begin{equation*}
\Vert \widetilde R_{\lambda }g\Vert _{\alpha ,\beta ;p} \leq C\sum_{r=2,p}
\Vert g\Vert_{\alpha ^{\prime },\beta ^{\prime };r,p} ; 
\end{equation*}

$\mathrm{(ii)}$ 
\begin{equation*}
\Vert \widetilde R_{\lambda }g\Vert _{0,\beta ;p} \leq C\sum_{r=2,p} \Bigl( %
T\wedge\frac{1}{\lambda }\Bigr)^{1/r} \Vert g\Vert_{0,\beta ;r,p} ; 
\end{equation*}

$\mathrm{(iii)}$ 
\begin{equation*}
\Vert \widetilde R_{\lambda }g(t^{\prime },\cdot)-\widetilde R_{\lambda
}g(t,\cdot)\Vert _{\alpha ^{\prime },\beta^{\prime };p} \leq C\sum_{r=2,p}
(t^{\prime 1/r} \Vert g\Vert_{\alpha ^{\prime },\beta ^{\prime };r,p} . 
\end{equation*}
\end{lemma}

\begin{proof}
(i). By Lemma 16 \cite{mip2}, 
\begin{equation*}
\int \big\vert h_{s,t}^{\lambda ,0}(x)\big\vert dx\leq Ce^{-\lambda (t-s)}. 
\end{equation*}%
Hence, for $0\leq u\leq t,\ r\geq 1$ 
\begin{equation}
\int_{u}^{t}\biggl(\int \big\vert h_{s,t}^{\lambda ,0}(x)\big\vert dx\biggr)%
^{r}ds\leq C(t-u).  \label{32}
\end{equation}

Using the estimate of Lemma \ref{lem6} (i), we have for $0\leq u\leq t,\
r\geq 1,\ j\geq 1$ 
\begin{eqnarray}
&&\int_{u}^{t}\biggl(\int \big\vert h_{s,t}^{\lambda ,j}(x)\big\vert dx%
\biggr)^{r}ds\leq  \notag \\
&&\qquad \leq C\int_{u}^{t}\biggl[e^{-c(2^{j\alpha }+\lambda
)(t-s)}\sum_{k\leq d_{0}}\bigl[2^{j\alpha }(t-s)\bigr]^{k}\biggr]^{r}ds\leq 
\notag \\
&&\qquad \leq C2^{-j\alpha }\int_{0}^{2^{j\alpha }(t-u)}e^{-crv}\biggl(%
\sum_{k\leq d_{0}}\upsilon ^{k}\biggr)^{r}d\upsilon \leq  \notag \\
&&\qquad \leq C2^{-j\alpha }\min \bigl[1,2^{j\alpha }(t-u)\bigr].  \label{33}
\end{eqnarray}%
These estimates, together with Lemma \ref{hl1}, imply 
\begin{eqnarray*}
\big\vert\widetilde{R}_{\lambda }g_{j}(t,x)\big\vert_{p} &\leq
&C\sum_{r=2,p}\Vert g_{j}\Vert _{r,p}\biggl\{\int_{0}^{t}\biggl(\int %
\big\vert h_{s,t}^{\lambda ,j}(x)\big\vert dx\biggr)^{r}ds\biggr\}^{1/r}\leq
\\
&\leq &C\sum_{r=2,p}2^{-j\alpha /r}\Vert g_{j}\Vert _{r,p}.
\end{eqnarray*}%
By Lemma \ref{lem4}, 
\begin{eqnarray*}
\Vert \widetilde{R}_{\lambda }g\Vert _{\alpha ,\beta ;p} &\leq &C\Bigl(\Vert 
\widetilde{R}_{\lambda }g_{0}\Vert _{p}+\sup_{j\geq 1}2^{j(\alpha +\beta
)}\Vert \widetilde{R}_{\lambda }g_{j}\Vert _{p}\Bigr)\leq \\
&\leq &C\biggl(\sum_{r=2,p}\Vert g_{0}\Vert _{r,p}+\sup_{j\geq 1}2^{j(\alpha
^{\prime }+\beta ^{\prime })}\sum_{r=2,p}\Vert g_{j}\Vert _{r,p}\biggr)\leq
\\
&\leq &C\sum_{r=2,p}\Vert g\Vert _{\alpha ^{\prime },\beta ^{\prime };r,p}.
\end{eqnarray*}

(ii). By Lemma \ref{hl1} and (\ref{22}), (\ref{24}) 
\begin{eqnarray*}
&&\big\vert\widetilde{R}_{\lambda }g(t,x)-\widetilde{R}_{\lambda
}g(t,x^{\prime })\big\vert_{p}= \\
&&\quad =\bigg\vert\int_{0}^{t}\int_{U}\int_{\mathbb{R}^{d}}G_{s,t}^{\lambda
}(y)[g(s,x-y,\upsilon )-g(s,x^{\prime }-y,\upsilon )]dyq(ds,d\upsilon )%
\bigg\vert_{p}\leq \\
&&\quad \leq C\sum_{r=2,p}\sup_{s\leq t,y}|g(s,x-y,\cdot )-g(s,x^{\prime
}-y,\cdot )|_{r,p}\biggl\{\int_{0}^{t}\biggl(\int G_{s,t}^{\lambda }(y)dy%
\biggr)^{r}ds\biggr\}^{1/r}\leq \\
&&\quad \leq C|x-x^{\prime }|^{\beta }\sum_{r=2,p}[g]_{\beta ;r,p}\biggl(%
\int_{0}^{t}e^{-\lambda (t-s)}ds\biggr)^{1/r}\leq \\
&&\quad \leq C|x-x^{\prime }|^{\beta }\sum_{r=2,p}\Bigl(T\wedge \frac{1}{%
\lambda }\Bigr)^{1/r}[g]_{\beta ;r,p}.
\end{eqnarray*}%
From this estimate and (\ref{25}) follows the assertion (ii).

(iii). By Lemma 16 \cite{mip2}, 
\begin{equation*}
\int \big\vert h_{s,t^{\prime }}^{\lambda ,0}(x)-h_{s,t}^{\lambda ,0}(x)%
\big\vert dx\leq C(1+\lambda )e^{-\lambda (t-s)}(t^{\prime }-t). 
\end{equation*}%
Hence, for $r\geq 1$ 
\begin{equation}
\int_{0}^{t}\biggl(\int \big\vert h_{s,t^{\prime }}^{\lambda
,0}(x)-h_{s,t}^{\lambda ,0}(x)\big\vert dx\biggr)^{r}ds\leq C(t^{\prime
}-t)^{r}.  \label{34}
\end{equation}

Using the estimate of Lemma \ref{lem6} (ii) with $\kappa =1/r$, we have for $%
r>1$ and $j\geq 1$ 
\begin{eqnarray}
&&\int_{0}^{t}\biggl(\int \big\vert h_{s,t^{\prime }}^{\lambda
,j}(x)-h_{s,t}^{\lambda ,j}(x)\big\vert dx\biggr)^{r}ds\leq  \notag \\
&&\quad \leq C2^{j\alpha }(t^{\prime }-t)\int_{0}^{t}e^{-cr(2^{j\alpha
}+\lambda )(t-s)}\biggl(\sum_{k\leq d_{0}}\bigl[2^{j\alpha }(t-s)\bigr]^{k}%
\biggr)^{r}ds\leq  \notag \\
&&\quad \leq C(t^{\prime }-t)\int_{0}^{\infty }e^{-cr\upsilon }\biggl(%
\sum_{k\leq d_{0}}\upsilon ^{k}\biggr)^{r}d\upsilon \leq  \label{35} \\
&&\quad \leq C(t^{\prime }-t).  \notag
\end{eqnarray}%
According to Lemma \ref{hl1} and estimates (\ref{32})--(\ref{35}), for $%
j\geq 0$ 
\begin{eqnarray*}
&&\big\vert\widetilde{R}_{\lambda }g_{j}(t^{\prime },x)-\widetilde{R}%
_{\lambda }g_{j}(t,x)\big\vert_{p}\leq \\
&&\quad \leq \bigg\vert\int_{t}^{t^{\prime }}\int_{U}\bigl[h_{s,t^{\prime
}}^{\lambda ,j}(\cdot )\ast g_{j}(s,\cdot ,\upsilon )\bigr](x)q(ds,d\upsilon
)\bigg\vert_{p}+ \\
&&\qquad +\bigg\vert\int_{0}^{t}\int_{U}\Bigl[\bigl(h_{s,t^{\prime
}}^{\lambda ,j}(\cdot )-h_{s,t}^{\lambda ,j}(\cdot )\bigr)\ast g_{j}(s,\cdot
,\upsilon )\Bigr](x)q(ds,d\upsilon )\bigg\vert_{p}\leq \\
&&\quad \leq C\sum_{r=2,p}\Vert g_{j}\Vert _{r,p}\bigg[\biggl(%
\int_{t}^{t^{\prime }}\Bigl(\int \big\vert h_{s,t^{\prime }}^{\lambda ,j}(x)%
\big\vert dx\Bigr)^{r}ds\biggr)^{1/r}+ \\
&&\qquad +\biggl(\int_{0}^{t}\Bigl(\int \big\vert h_{s,t^{\prime }}^{\lambda
,j}(x)-h_{s,t}^{\lambda ,j}(x)\big\vert dx\Bigr)^{r}ds\biggr)^{1/r}\bigg]\leq
\\
&&\quad \leq C\sum_{r=2,p}(t^{\prime }-t)^{1/r}\Vert g_{j}\Vert _{r,p}.
\end{eqnarray*}%
Finally, by Lemma \ref{lem4}, 
\begin{eqnarray*}
&&\Vert \widetilde{R}_{\lambda }g(t^{\prime },\cdot )-\widetilde{R}_{\lambda
}g(t,\cdot )\Vert _{\alpha ^{\prime },\beta ^{\prime };p}\leq C\bigg(\Vert 
\widetilde{R}_{\lambda }g_{0}(t^{\prime },\cdot )-\widetilde{R}_{\lambda
}g_{0}(t,\cdot )\Vert _{p}+ \\
&&\qquad +\sup_{j\geq 1}2^{j(\alpha ^{\prime }+\beta ^{\prime })}\Vert 
\widetilde{R}_{\lambda }g_{j}(t^{\prime },\cdot )-\widetilde{R}_{\lambda
}g_{j}(t,\cdot )\Vert _{p}\bigg)\leq \\
&&\quad \leq C\biggl(\sum_{r=2,p}(t^{\prime }-t)^{1/r}\Vert g_{0}\Vert
_{r,p}+\sup_{j\geq 1}2^{j(\alpha ^{\prime }+\beta ^{\prime
})}\sum_{r=2,p}(t^{\prime 1/r}\Vert g_{j}\Vert _{r,p}\biggr)\leq \\
&&\quad \leq C\sum_{r=2,p}(t^{\prime }-t)^{1/r}\Vert g\Vert _{\alpha
^{\prime },\beta ^{\prime };r,p}.
\end{eqnarray*}%
The lemma is proved.
\end{proof}

\begin{lemma}
\label{lem8} Let $\alpha \in (0,2],\ \beta \in (0,1)$, $p\geq 1,$ $f\in
C_{p}^{0,\beta }(H)$, and let assumptions~\textbf{A1, A2} be satisfied.

Then there is a constant $C$ depending only on $\alpha ,\beta
,p,d,\mu,C^{(\alpha )},T$ such that the following estimates hold:

$\mathrm{(i)}$ 
\begin{equation*}
\Vert R_{\lambda }f\Vert_{\alpha ,\beta ;p} \leq C \Vert f\Vert _{0,\beta
;p} ; 
\end{equation*}

$\mathrm{(ii)}$ 
\begin{equation*}
\Vert R_{\lambda }f\Vert_{0 ,\beta ;p} \leq C \Bigl( T\wedge\frac{1}{\lambda 
}\Bigr) \Vert f\Vert _{0,\beta ;p} ; 
\end{equation*}

$\mathrm{(iii)}$\ \ for $0\leq t\leq t^{\prime }\leq T$ and $\nu \in
(0,\alpha )$ 
\begin{equation*}
\Vert R_{\lambda }f(t^{\prime },\cdot )-R_{\lambda }f(t,\cdot )\Vert _{\nu
,\beta ;p}\leq C(t^{\prime }-t)^{1-\frac{\nu }{\alpha }}\Vert f\Vert
_{0,\beta ;p}. 
\end{equation*}
\end{lemma}

\begin{proof}
(i). Using the Minkowsky inequality and estimates (\ref{32}), (\ref{33})
with $r=1$, we have for $j\geq 0$ 
\begin{eqnarray*}
|R_{\lambda }f_{j}(t,x)|_{p} &\leq &C\Vert f_{j}\Vert _{p}\int_{0}^{t}\int %
\big\vert h_{s,t}^{\lambda ,j}(x)\big\vert dxds\leq \\
&\leq &C2^{-j\alpha }\Vert f_{j}\Vert _{p}.
\end{eqnarray*}%
By Lemma \ref{lem4}, 
\begin{eqnarray*}
\Vert R_{\lambda }f\Vert _{\alpha ,\beta ;p} &\leq &C\Bigl(\Vert R_{\lambda
}f_{0}\Vert _{p}+\sup_{j\geq 1}2^{j(\alpha +\beta )}\Vert R_{\lambda
}f_{j}\Vert _{p}\Bigr)\leq \\
&\leq &C\Bigl(\Vert f_{0}\Vert _{p}+\sup_{j\geq 1}2^{j\beta }\Vert
f_{j}\Vert _{p}\Bigr)\leq \\
&\leq &C\Vert f\Vert _{0,\beta ;p}.
\end{eqnarray*}

(ii). According to the Minkowsky inequality and (\ref{22}), (\ref{23}), for $%
(t,x),(t,x^{\prime })\in H$ 
\begin{eqnarray*}
&&|R_{\lambda }f(t,x)-R_{\lambda }(t,x^{\prime })|_{p}= \\
&&\quad =\bigg\vert\int_{0}^{t}\int G_{s,t}^{\lambda }(y)\bigl[%
f(s,x-y)-f(s,x^{\prime }-y)\bigr]dyds\bigg\vert_{p}\leq \\
&&\quad \leq \sup_{s\leq t,y}|f(s,x-y)-f(s,x^{\prime
}-y)|_{p}\int_{0}^{t}\int G_{s,t}^{\lambda }(y)dyds\leq \\
&&\quad \leq |x-x^{\prime }|^{\beta }[f]_{\beta ;p}\int_{0}^{t}e^{-\lambda
(t-s)}ds\leq \\
&&\quad \leq |x-x^{\prime }|^{\beta }\Bigl(T\wedge \frac{1}{\lambda }\Bigr)%
\lbrack f]_{\beta ;p}.
\end{eqnarray*}%
This estimate, together with (\ref{26}), implies the assertion (ii).

(iii). Using Lemma \ref{lem6} and H\"{o}lder's inequality, we obtain for $%
\kappa \in (0,1)$ and $j\geq 1$ 
\begin{eqnarray*}
&&\int_{t}^{t^{\prime }}\int \big\vert h_{s,t}^{\lambda ,j}(x)\big\vert %
dxds\leq C\int_{t}^{t^{\prime }}e^{-c(2^{j\alpha }+\lambda
)(t-s)}\sum_{k\leq d_{0}}\bigl[2^{j\alpha }(t-s)\bigr]^{k}ds\leq \\
&&\quad \leq C2^{-j\alpha }\int_{0}^{2^{j\alpha }(t^{\prime }-t)}e^{-cs}%
\biggl(\sum_{k\leq d_{0}}s^{k}\biggr)ds\leq \\
&&\quad \leq C2^{-j\alpha }\bigl[2^{j\alpha }(t^{\prime }-t)\bigr]^{\kappa }%
\biggl\{\int_{0}^{\infty }\biggl[e^{-cs}\sum_{k\leq d_{0}}s^{k}\biggr]%
^{1/(1-\kappa )}ds\biggr\}^{1-\kappa }\leq \\
&&\quad \leq C2^{-j\alpha (1-\kappa )}(t^{\prime }-t)^{\kappa }
\end{eqnarray*}%
and 
\begin{eqnarray*}
&&\int_{0}^{t}\int \big\vert h_{s,t^{\prime }}^{\lambda
,j}(x)-h_{s,t}^{\lambda ,j}(x)\big\vert dxds\leq \\
&&\quad \leq C\bigl[2^{j\alpha }(t^{\prime }-t)\bigr]^{\kappa
}\int_{0}^{t}e^{-c(2^{j\alpha }+\lambda )(t-s)}\sum_{k\leq d_{0}}\bigl[%
2^{j\alpha }(t-s)\bigr]^{k}ds\leq \\
&&\quad \leq C2^{-j\alpha (1-\kappa )}(t^{\prime }-t)^{\kappa }.
\end{eqnarray*}

According to (\ref{32}) and (\ref{34}), the same estimates hold in the case $%
j=0$. Therefore, by Minkowsky's inequality, for $\kappa \in (0,1)$ and $%
j\geq 0.$ 
\begin{eqnarray*}
&&\big\vert R_{\lambda }f_{j}(t^{\prime },x)-R_{\lambda }f_{j}(t,x)\big\vert%
_{p}\leq \bigg\vert\int_{t}^{t^{\prime }}\bigl[h_{s,t^{\prime }}^{\lambda
,j}\ast f_{j}(s,\cdot )\bigr](x)ds\bigg\vert_{p}+ \\
&&\qquad +\bigg\vert\int_{0}^{t}\int \bigl[\bigl(h_{s,t^{\prime }}^{\lambda
,j}-h_{s,t}^{\lambda ,j}\bigr)\ast f_{j}(s,\cdot )\bigr](x)ds\bigg\vert%
_{p}\leq \\
&&\quad \leq C\Vert f_{j}\Vert _{p}\biggl(\int_{t}^{t^{\prime }}\int %
\big\vert h_{s,t^{\prime }}^{\lambda ,j}(x)\big\vert dxds+\int_{0}^{t}\int %
\big\vert h_{s,t^{\prime }}^{\lambda ,j}(x)-h_{s,t}^{\lambda ,j}(x)\big\vert %
dxds\biggr)\leq \\
&&\quad \leq C2^{-j\alpha (1-\kappa )}(t^{\prime }-t)^{\kappa }\Vert
f_{j}\Vert _{p}.
\end{eqnarray*}%
This estimate with $\kappa =1-\nu /\alpha $ and Lemma \ref{lem4} imply 
\begin{eqnarray*}
&&\Vert R_{\lambda }f(t^{\prime },\cdot )-R_{\lambda }f(t,\cdot )\Vert _{\nu
,\beta ;p}\leq \\
&&\quad \leq C\Bigl(\Vert R_{\lambda }f_{0}(t^{\prime },\cdot )-R_{\lambda
}f_{0}(t,\cdot )\Vert _{p}+\sup_{j\geq 1}2^{j(\nu +\beta )}\Vert R_{\lambda
}f_{j}(t^{\prime },\cdot )-R_{\lambda }f_{j}(t,\cdot )\Vert _{p}\Bigr)\leq \\
&&\quad \leq C(t^{\prime 1-\nu /\alpha }\Bigl(\Vert f_{0}\Vert
_{p}+\sup_{j\geq 1}2^{j\beta }\Vert f_{j}\Vert _{p}\Bigr)\leq \\
&&\quad \leq C(t^{\prime }-t)^{1-\nu /\alpha }\Vert f\Vert _{0,\beta ;p}.
\end{eqnarray*}%
The lemma is proved.
\end{proof}

\subsection{Proof of Theorem \protect\ref{main}}

Let 
\begin{equation*}
u=R_{\lambda }f+\widetilde{R}_{\lambda }g. 
\end{equation*}%
According to Lemmas \ref{lem7} and \ref{lem8}, the function $u$ belongs to
the space $C_{p}^{\alpha ,\beta }(H)$ and satisfies all the required
estimates (we take $\nu =\alpha ^{\prime }+\beta ^{\prime }-\beta =\alpha
(1-1/p)$ in Lemma \ref{lem8} (iii) and notice that, by Lemma \ref{lem4}, the
norms $\Vert \cdot \Vert _{\alpha ^{\prime },\beta ^{\prime };p}$ and $\Vert
\cdot \Vert _{\nu ,\beta ;p}$ are equivalent). By Corollary \ref{lu}, the
equation (\ref{eq1}) has at most one solution in the space $C_{p}^{\alpha
,\beta }(H).$ Hence, it remains to prove that $u$ is a solution to (\ref{eq1}%
).

Let 
\begin{eqnarray*}
f_{n}(t,\cdot ) &=&f(t,\cdot )\ast \psi +\sum_{j=1}^{n}f(t,\cdot )\ast
\varphi _{j}, \\
g_{n}(t,\cdot ,\upsilon ) &=&g(t,\cdot ,\upsilon )\ast \psi
+\sum_{j=1}^{n}g(t,\cdot ,\upsilon )\ast \varphi _{j},\quad n\geqslant 1,
\end{eqnarray*}%
where the functions $\psi ,\varphi _{j},\ j\geqslant 1$, are defined by (\ref%
{28}) and (\ref{29}). By Lemma \ref{l0}, the function 
\begin{equation*}
u_{n}=R_{\lambda }f_{n}+\widetilde{R}_{\lambda }g_{n} 
\end{equation*}%
is a unique solution in $C_{p}^{\infty }(H)$ to the equation (\ref{eq1})
with $f,g$ replaced by $f_{n},g_{n}$.

Let $\nu \in (0,\beta ^{\prime })$ be such that 
\begin{equation*}
\beta _{\nu }=\alpha ^{\prime }+\nu -\alpha \Bigl(1-\frac{1}{p}\Bigr)>0. 
\end{equation*}%
Using the estimates of Lemmas \ref{lem7}, \ref{lem8} and Lemma \ref{lem5},
we get 
\begin{eqnarray}
\Vert u-u_{n}\Vert _{\alpha ,\beta _{\nu };p} &\leqslant &\Vert R_{\lambda
}(f-f_{n})\Vert _{\alpha ,\beta _{\nu };p}+\Vert \widetilde{R}_{\lambda
}(g-g_{n})\Vert _{\alpha ,\beta _{\nu };p}\leqslant  \notag \\
&\leqslant &C\biggl(\Vert f-f_{n}\Vert _{0,\beta _{\nu
};p}+\sum_{r=2,p}\Vert g-g_{n}\Vert _{\alpha ^{\prime },\nu ;r,p}\biggr)%
\rightarrow 0  \label{36}
\end{eqnarray}%
as $n\rightarrow \infty $.

Let us introduce the function 
\begin{equation*}
\overline{u}(t,x)=\int_{0}^{t}(Au-\lambda
u+f)(s,x)ds+\int_{0}^{t}\int_{U}g(s,x,\upsilon )q(ds,d\upsilon ). 
\end{equation*}%
Then $\mathbf{P}$-a.s. for each $(t,x)\in H$ 
\begin{eqnarray*}
(\overline{u}-u_{n})(t,x) &=&\int_{0}^{t}\bigl[(A-\lambda )(u-u_{n})+f-f_{n}%
\bigr](s,x)ds+ \\
&&+\int_{0}^{t}\int_{U}(g-g_{n})(s,x,\upsilon )q(ds,d\upsilon ).
\end{eqnarray*}%
By Minkowsky's inequality and Lemma \ref{hl0}, 
\begin{eqnarray}
\Vert \overline{u}-u_{n}\Vert _{p} &\leqslant &C\bigg(\Vert (A-\lambda
)(u-u_{n})\Vert _{p}+  \notag \\
&&+\Vert f-f_{n}\Vert _{p}+\sum_{r=2,p}\Vert g-g_{n}\Vert _{r,p}\bigg).
\label{37}
\end{eqnarray}%
By Lemma 20 \cite{mip2}, for each $\upsilon \in C_{p}^{\alpha ,\kappa }(H),\
\kappa \in (0,1)$, 
\begin{eqnarray*}
\Vert A\upsilon \Vert _{0,\kappa ;p} &=&\Big\Vert\mathcal{F}^{-1}\bigl\{\psi
^{(\alpha )}(t,\xi )(1+|\xi |^{\alpha })^{-1}\mathcal{F}(\upsilon +\partial
^{\alpha }\upsilon )\bigr\}\Big\Vert_{0,\kappa ;p}\leqslant \\
&\leqslant &C\Vert \upsilon +\partial ^{\alpha }\upsilon \Vert _{0,\kappa
;p}\leqslant C\Vert \upsilon \Vert _{\alpha ,\kappa ;p}.
\end{eqnarray*}%
Hence, according to (\ref{36}), 
\begin{equation*}
\Vert A(u-u_{n})\Vert _{0,\beta _{\nu };p}\leqslant C\Vert u-u_{n}\Vert
_{\alpha ,\beta _{\nu };p}\rightarrow 0 
\end{equation*}%
as $n\rightarrow \infty $, and by (\ref{37}), $\Vert \overline{u}-u_{n}\Vert
_{p}\rightarrow 0$ as $n\rightarrow \infty $. Thus, for each $(t,x)\in H$,
we have $u(t,x)=\overline{u}(t,x)$ $\mathbf{P}$-a.s.

The theorem is proved.

\section{Appendix}

In the following lemma, we prove the existence of smooth modifications of
stochastic integrals.

\begin{lemma}
\label{cl1}Assume $g\in {C}_{l,p}^{\infty}(H\times U)$, $l=2,p$. Then:

$\mathrm{a)}$ There is a $\mathcal{P}(\mathbb{F)}\mathcal{\otimes B}(\mathbf{%
R}^{d})\otimes \mathcal{U}$-measurable function $\widetilde{g}(s,x,\upsilon
) $ such that $d\Pi dsd\mathbf{P}$-a.e. 
\begin{equation}
\partial _{x}^{\gamma }g(s,x,\upsilon )=\partial _{x}^{\gamma }\widetilde{g}%
(s,x,\upsilon )  \label{cf0}
\end{equation}
for all $x\in \mathbf{R}^{d},\gamma \in \mathbf{N}_{0}^{d}$. In addition,
for any $R>0$ and $\gamma \in \mathbf{N}_{0}^{d},l=2,p,$%
\begin{equation}
\mathbf{E}\int_{0}^{T}\int_{U}\sup_{|x|\leq R}|\partial _{x}^{\gamma }%
\widetilde{g}(s,x,\upsilon )|^{l}\Pi (d\upsilon )ds<\infty .  \label{cf01}
\end{equation}

$\mathrm{b)}$ There is an $\mathcal{O(\mathbb{F)\otimes }B}(\mathbf{R}^{d})$%
-measurable real-valued function $M(t,x)$ $\mathbf{P}$-a.s. cadlag in $t$
and smooth in $x$ and, for each $x\in \mathbf{R}^{d},\gamma \in \mathbf{N}%
_{0}^{d},\mathbf{P} $-a.s. 
\begin{equation}
\partial _{x}^{\gamma }M(t,x)=\int_{0}^{t}\int_{U}\partial _{x}^{\gamma
}g(s,x,\upsilon )q(ds,d\upsilon )=\int_{0}^{t}\int_{U}\partial _{x}^{\gamma }%
\widetilde{g}(s,x,\upsilon )q(ds,d\upsilon )  \label{cf2}
\end{equation}%
for all $t\in \lbrack 0,T]$. In addition, for any $R>0$ and $\gamma \in 
\mathbf{N}_{0}^{d}$%
\begin{equation}
\mathbf{E[}\sup_{t\leq T,|x|\leq R}|\partial _{x}^{\gamma
}M(t,x)|^{p}]<\infty ,  \label{cf4}
\end{equation}%
and%
\begin{equation}
\sup_{x}\mathbf{E[}\sup_{t\leq T}|\partial _{x}^{\gamma }M(t,x)|^{p}]<\infty
.  \label{cf5}
\end{equation}
\end{lemma}

\begin{proof}
Obviously, for each $R>0,\gamma \in \mathbf{N}_{0}^{d},$%
\begin{equation*}
\mathbf{E}\int_{0}^{T}\int_{U}\int_{|y|<R}|\partial _{y}^{\gamma
}g(s,y,\upsilon )|^{l}dy\Pi (d\upsilon )ds<\infty , 
\end{equation*}%
$l=2,p$. Define a sequence of $\mathcal{P}(\mathbb{F})\mathcal{\otimes B}(%
\mathbf{R}^{d})\otimes \mathcal{U}$-measurable functions%
\begin{eqnarray*}
g_{n}(t,x,\upsilon ) &=&n\int_{t_{n}}^{t}g(s,x,\upsilon )ds,n\geq 1, \\
g_{n}^{\gamma }(t,x,\upsilon ) &=&n\int_{t_{n}}^{t}\partial _{x}^{\gamma
}g(s,x,\upsilon )ds,\gamma \in \mathbf{N}_{0}^{d},n\geq 1,
\end{eqnarray*}%
where $t_{n}=(t-1/n)\vee 0$. For each $\gamma \in \mathbf{N}^d_{0}$ we have $%
\mathbf{P}$-a.s. for all $t\geq 0,\varphi \in C_{0}^{\infty }(\mathbf{R}%
^{d}) $%
\begin{eqnarray}
\int_{\mathbf{R}^{d}}g_{n}(t,y,\upsilon )\partial _{y}^{\gamma }\varphi
(y)dy &=&\int_{\mathbf{R}^{d}}n\int_{t_{n}}^{t}g(s,y,\upsilon )ds\partial
^{\gamma }\varphi (y)dy  \label{cf1} \\
&=&(-1)^{|\gamma |}\int_{\mathbf{R}^{d}}n\int_{t_{n}}^{t}\partial ^{\gamma
}g(s,y,\upsilon )ds\varphi (y)dy.  \notag
\end{eqnarray}

Let $w\in C_{0}^{\infty }(\mathbf{R}^{d})$ be a non-negative function such
that $w(x)=0$ if $|x|\geq 1$ and $\int w(x)dx=1$. We define $\mathbf{P}$%
-a.s. continuous in $t$ and smooth in $x$ functions%
\begin{eqnarray*}
g_{n,\varepsilon }(t,x,\upsilon ) &=&n\int_{t_{n}}^{t}\int_{\mathbf{R}%
^{d}}g(s,y,\upsilon )w_{\varepsilon }(x-y)dsdy, \\
g_{n,\varepsilon }^{\gamma }(t,x,\upsilon ) &=&n\int_{t_{n}}^{t}\int_{%
\mathbf{R}^{d}}\partial _{y}^{\gamma }g(s,y,\upsilon )w_{\varepsilon
}(x-y)dsdy,
\end{eqnarray*}%
where $w_{\varepsilon }(x)=\varepsilon ^{-d}w(x/\varepsilon ),x\in \mathbf{R}%
^{d},\gamma \in \mathbf{N}_{0}^{d}$. According to (\ref{cf1}), $\mathbf{P}$%
-a.s. for all $t\geq 0,x\in \mathbf{R}^{d},\gamma \in \mathbf{N}_{0}^{d},$%
\begin{eqnarray*}
\partial _{x}^{\gamma }g_{n,\varepsilon }(t,x,\upsilon ) &=&\int_{\mathbf{R}%
^{d}}n\int_{t_{n}}^{t}g(s,y,\upsilon )ds\partial _{x}^{\gamma
}w_{\varepsilon }(x-y)dy \\
&=&\int_{\mathbf{R}^{d}}n\int_{t_{n}}^{t}\partial _{x}^{\gamma
}g(s,x-y,\upsilon )dsw_{\varepsilon }(y)dy=g_{n,\varepsilon }^{\gamma }(t,x).
\end{eqnarray*}%
Therefore for each $R>0$ and $l=2,p,$%
\begin{eqnarray*}
&&\mathbf{E}\int_{0}^{T}\int_{|x|<R}\int_{U}|\partial _{x}^{\gamma
}g_{n,\varepsilon }(t,x,\upsilon )-\partial _{x}^{\gamma }g_{n^{\prime
},\varepsilon ^{\prime }}(t,x,\upsilon )|^{l}\Pi (d\upsilon )dxdt \\
&\leq &C[\int_{0}^{T}\int_{|x|<R}|g_{n,\varepsilon }^{\gamma }(t,x,\cdot
)-\partial ^{\gamma }g(t,x,\cdot )|_{l,l}^{l}dtdx \\
&&+\int_{0}^{T}\int_{|x|<R}|g_{n^{\prime },\varepsilon ^{\prime }}^{\gamma
}(t,x,\cdot )-\partial ^{\gamma }g(t,x,\cdot )|_{l,l}^{l}dtdx]\rightarrow 0
\end{eqnarray*}%
as $\varepsilon ,\varepsilon ^{\prime }\rightarrow 0$ and $n,n^{\prime
}\rightarrow \infty .$ By the Sobolev embedding theorem, for each $%
R>0,\gamma \in \mathbf{N}_{0}^{d}$,$l=2,p,$%
\begin{equation*}
\mathbf{E}\int_{0}^{T}\int_{U}\sup_{|x|\leq R}|\partial _{x}^{\gamma
}g_{n,\varepsilon }(t,x,\upsilon )-\partial _{x}^{\gamma }g_{n^{\prime
},\varepsilon ^{\prime }}(t,x,\upsilon )|^{l}\Pi (dz)dt\rightarrow 0 
\end{equation*}%
as $\varepsilon ,\varepsilon ^{\prime }\rightarrow 0$ and $n,n^{\prime
}\rightarrow \infty .$ Let $\widetilde{g}_{n}=g_{n,1/n}$ and choose a
subsequence $n_{k}\uparrow \infty $ as $k\rightarrow \infty $ such that for $%
l=2,p,$%
\begin{equation*}
\mathbf{E}\int_{0}^{T}\int_{U}\sup_{|x|<k,|\gamma |\leq k}|\partial
_{x}^{\gamma }\widetilde{g}_{n_{k+1}}(t,x,\upsilon )-\partial _{x}^{\gamma }%
\widetilde{g}_{n_{k}}(t,x,\upsilon )|^{l}\Pi (d\upsilon )dt\leq 2^{-kp}. 
\end{equation*}%
Then the function%
\begin{equation*}
\widetilde{g}(t,x,\upsilon )=\left\{ 
\begin{array}{cc}
\displaystyle\lim_{k}\widetilde{g}_{n_{k}}(t,x) & \text{if }%
\sum_{k}\sup_{|x|<k,|\gamma |\leq k}|\partial _{x}^{\gamma }\widetilde{g}%
_{n_{k+1}}-\partial _{x}^{\gamma }\widetilde{g}_{n_{k}}|(t,x,\upsilon
)<\infty , \\ 
0 & \text{otherwise,}%
\end{array}%
\right. 
\end{equation*}%
is $\mathcal{P}(\mathbb{F})\mathcal{\otimes B}(\mathbf{R}^{d})\otimes 
\mathcal{U}$-measurable and satisfies (\ref{cf0}) and (\ref{cf01}).

b) By Theorem 5.44 in \cite{jacod}, there are $\mathcal{O}(\mathbb{F}%
)\otimes \mathcal{B}(\mathbf{R}^{d})$-measurable functions $\widetilde{M}%
^{\gamma }(t,x),\gamma \in \mathbf{N}_{0}^{d}$ such that for each $x\in 
\mathbf{R}^{d} $ $\mathbf{P}$-a.s. 
\begin{equation*}
\widetilde{M}^{\gamma }(t,x)=\int_{0}^{t}\int_{U}\partial _{x}^{\gamma }%
\widetilde{g}(s,x,\upsilon )q(ds,d\upsilon )=\int_{0}^{t}\int_{U}\partial
_{x}^{\gamma }g(s,x,\upsilon )q(ds,d\upsilon ),t\in \lbrack 0,T], 
\end{equation*}%
where $\mathcal{O}(\mathbb{F})$ is the $\sigma $-algebra of well measurable
subsets of $[0,T]\times \Omega $. By the Burk\-holder--Davis--Gundy
inequality and Lemma \ref{hl0}, for each $R>0,t\in \lbrack 0,T],$%
\begin{eqnarray*}
\mathbf{E}\int_{|x|\leq R}\sup_{t\leq T}|\widetilde{M}^{\gamma }(t,x)|^{p}dx
&\leq &C\int_{|x|\leq R}\sum_{l=2,p}\mathbf{E}\left(
\int_{0}^{T}\int_{U}|\partial _{x}^{\gamma }g(s,x,\upsilon )|^{l}\Pi
(d\upsilon )ds\right) ^{p/l}dx \\
&\leq &C\sum_{l=2,p}\sup_{0\leq s\leq T,x}||\partial _{x}^{\gamma
}g(s,x,\cdot )||_{l,p}^{p}<\infty .
\end{eqnarray*}%
We define $\mathbf{P}$-a.s. cadlag in $t$ and smooth in $x$ functions%
\begin{equation*}
\widetilde{M}_{\varepsilon }^{\gamma }(t,x)=\int_{\mathbf{R}^{d}}\widetilde{M%
}^{\gamma }(t,y)w_{\varepsilon }(x-y)dx,\gamma \in \mathbf{N}_{0}^{d}. 
\end{equation*}%
By the stochastic Fubini theorem (see \cite{mik3}), for every $x\in \mathbf{R%
}^{d}$ we have $\mathbf{P}$-a.s. for all $t\in \lbrack 0,T]$ and $\gamma \in 
\mathbf{N}_{0}^{d},$%
\begin{eqnarray*}
\widetilde{M}_{\varepsilon }^{\gamma }(t,x) &=&\int \widetilde{M}^{\gamma
}(t,y)w_{\varepsilon }(x-y)dx=\int_{0}^{t}\int_{U}\int \partial _{x}^{\gamma
}\widetilde{g}(s,y,\upsilon )w_{\varepsilon }(x-y)dyq(ds,d\upsilon ) \\
&=&\int_{0}^{t}\int_{U}\int_{\mathbf{R}^{d}}\partial _{x}^{\gamma
}g(s,y,\upsilon )w_{\varepsilon }(x-y)dyq(ds,d\upsilon ), \\
\partial _{x}^{\gamma }\widetilde{M}_{\varepsilon }^{0}(t,x) &=&\int_{%
\mathbf{R}^{d}}\widetilde{M}^{0}(t,y)\partial _{x}^{\gamma }w_{\varepsilon
}(x-y)dx=(-1)^{|\gamma |}\int_{\mathbf{R}^{d}}\widetilde{M}^{0}(t,y)\partial
_{y}^{\gamma }w_{\varepsilon }(x-y)dx \\
&=&\int_{0}^{t}\int_{U}\int \partial _{x}^{\gamma }\widetilde{g}%
(s,y,\upsilon )w_{\varepsilon }(x-y)dyq(ds,d\upsilon ) \\
&=&\int_{0}^{t}\int_{U}\int \partial _{x}^{\gamma }g(s,y,\upsilon
)w_{\varepsilon }(x-y)dyq(ds,d\upsilon ).
\end{eqnarray*}%
Also, denoting%
\begin{eqnarray*}
g^{\gamma ,\varepsilon }(s,x,\upsilon ) &=&\int \partial _{y}^{\gamma
}g(s,y,\upsilon )w_{\varepsilon }(x-y)dy \\
&=&\int \partial _{x}^{\gamma }g(s,x-y,\upsilon )w_{\varepsilon }(y)dy,
\end{eqnarray*}%
we have, by Lemma \ref{hl0}, for each $R>0$%
\begin{eqnarray*}
&&\mathbf{E}\int_{|x|<R}\sup_{t\leq T}|\partial _{x}^{\gamma }\widetilde{M}%
_{\varepsilon }^{0}(t,x)-\partial _{x}^{\gamma }\widetilde{M}_{\varepsilon
^{\prime }}^{0}(t,x)|^{p}dx \\
&=&\int_{|x|<R}\mathbf{E}\sup_{t\leq T}|\partial _{x}^{\gamma }\widetilde{M}%
_{\varepsilon }^{0}(t,x)-\partial _{x}^{\gamma }\widetilde{M}_{\varepsilon
^{\prime }}^{0}(t,x)|^{p}dx \\
&\leq &C\mathbf{E}\int_{0}^{T}\sum_{l=2,p}\int_{|x|<R}\bigg(%
\int_{U}|g^{\gamma ,\varepsilon }(s,x,\upsilon )-g^{\gamma ,\varepsilon
^{\prime }}(s,x,\upsilon )|^{l}\Pi (d\upsilon )\bigg)^{\frac{p}{l}%
}dxds\rightarrow 0
\end{eqnarray*}%
as $\varepsilon ,\varepsilon ^{\prime }\rightarrow 0$. By the Sobolev
embedding theorem, for each $R>0$ and $\gamma \in \mathbf{N}_{0}^{d},$%
\begin{equation*}
\mathbf{E}\sup_{|x|\leq R,t\leq T}|\partial _{x}^{\gamma }\widetilde{M}%
_{\varepsilon }^{0}(t,x)-\partial _{x}^{\gamma }\widetilde{M}_{\varepsilon
^{\prime }}^{0}(t,x)|^{p}\rightarrow 0 
\end{equation*}%
as $\varepsilon ,\varepsilon ^{\prime }\rightarrow 0$. Therefore, there is
an $\mathcal{O}(\mathbb{F})\otimes \mathcal{B}(\mathbf{R}^{d})$-measurable
function $M(t,x)$ which is $\mathbf{P}$-a.s. cadlag in $t$ and smooth in $x$%
, satisfies 
(40)--(42) hold and for each $R>0,\gamma \in \mathbf{N}_{0}^{d}$%
\begin{equation*}
\mathbf{E}\sup_{|x|\leq R,t\leq T}|\partial _{x}^{\gamma }\widetilde{M}%
_{\varepsilon }^{0}(t,x)-\partial _{x}^{\gamma }M(t,x)|^{p}\rightarrow 0 
\end{equation*}%
as $\varepsilon \rightarrow 0$. The lemma is proved.
\end{proof}

\end{document}